\documentclass[12pt]{article}
\usepackage{epsfig}
\usepackage{times}
\usepackage{amsmath,amsthm,amsfonts,amssymb,mathrsfs}

\newcommand{\bproof}{\begin{proof}}
\newcommand{\eproof}{\end{proof}}

\begin{document}
\bibliographystyle{plain}
\makeatletter

\renewcommand{\theequation}{\thesection.\arabic{equation}}
\numberwithin{equation}{section}
\renewcommand{\bar}{\overline}
\newtheorem{theorem}{Theorem}[section]
\newtheorem{question}{Question}[section]
\newtheorem{proposition}{Proposition}[section]
\newtheorem{lemma}{Lemma}[section]
\newtheorem{corollary}{Corollary}[section]
\newtheorem{definition}{Definition}[section]
\newtheorem{problem}{\em Problem}[section]
\newtheorem{remark}{Remark}[section]
\newtheorem{example}{Example}[section]
\newtheorem{case}{Case}[section]
\newtheorem{assumption}{Assumption}[section]

\renewcommand\proofname{\bf Proof} 
\def\eop{$\rule{1.3ex}{1.3ex}$}
\renewcommand\qedsymbol\eop  
\numberwithin{equation}{section}
\makeatletter

\newcommand{\lom}{(\log M)^{3/2+\delta}}

\newcommand{\XXX}{{\bf XXX~}}
\newcommand{\beq}{\begin{equation}} \newcommand{\eeq}{\end{equation}}
\newcommand{\eigv}{{v}}
\newcommand{\bz}{{\bf z}}
\newcommand{\bx}{{\bf x}}
\newcommand{\bt}{{\bf t}}
\newcommand{\tn}{{|\hspace{-.03truecm} | \hspace{-.03truecm}|}}
\newcommand{\bi}{\begin{itemize}}
\newcommand{\be}{\begin{enumerate}}
\newcommand{\ei}{\end{itemize}}
\newcommand{\ee}{\end{enumerate}}
\newcommand{\calH}{{\cal H}}
\newcommand{\E}{{\mathbb E}}
\newcommand{\Prob}{{\mathbb P}}
\newcommand{\Om}{{\Theta}}
\newcommand{\om}{{\theta}}
\newcommand{\gam}{{\gamma}}
\newcommand{\Gam}{{\Gamma}}
\newcommand{\homega}{{\hat \om}}
\newcommand{\bS}{{\mathbb S}}
\newcommand{\R}{{\mathbb R}}
\newcommand{\N}{{\mathbb N}}
\newcommand{\GD}{{\mathcal N}}
\newcommand{\calE}{{\cal E}}
\newcommand{\calG}{{\cal G}}
\newcommand{\calV}{{\cal V}}
\newcommand{\calK}{{\cal K}}
\newcommand{\calN}{{\cal N}}
\newcommand{\calU}{{\cal U}}
\newcommand{\calT}{{\cal T}}
\newcommand{\calY}{{\cal Y}}
\newcommand{\calO}{{\cal O}}
\newcommand{\calX}{{\cal X}}
\newcommand{\calW}{{\cal W}}
\newcommand{\W}{{\cal W}}
\newcommand{\G}{{\cal G}}
\newcommand{\K}{{\cal K}}
\newcommand{\OO}{{\bf O}}
\newcommand{\hK}{{\hat K}}
\newcommand{\X}{{\cal X}}
\newcommand{\M}{{\cal M}}
\newcommand{\KG}{{\cal K(\calG)}}
\newcommand{\lam}{{\lambda}}
\newcommand{\calM}{{\cal M}}
\newcommand{\calA}{{\cal A}}
\newcommand{\calB}{{\cal B}}
\newcommand{\calL}{{\cal L}}
\newcommand{\calD}{{\cal D}}
\newcommand{\calR}{{\cal R}}
\newcommand{\pp}{{\cal P}}
\newcommand{\hc}{{\hat c}}
\newcommand{\ck}{{c_K}}
\newcommand{\hL}{{\hat L}}
\newcommand{\tL}{{\bar L}}
\newcommand{\sK}{{\SSS K}}
\newcommand{\hg}{{g_K}}
\newcommand{\tf}{{f_K}}
\newcommand{\hy}{{y_K}}
\newcommand{\haty}{{\hat y}}
\newcommand{\hG}{{\hat \Gam}}
\newcommand{\vt}{{\vec t}}
\newcommand{\vv}{{\vec v}}
\newcommand{\lb}{{\langle}}
\newcommand{\rb}{{\rangle}}
\newcommand{\by}{{\bf y}}
\newcommand{\btau}{{\bf \tau}}
\newcommand{\bu}{{\bf u}}
\newcommand{\bv}{{\bf v}}
\newcommand{\tby}{\tilde{{\bf y}}}
\newcommand{\Sb}{{\bf S}}
\newcommand{\Mb}{{\bf M}}
\newcommand{\Ob}{{\bf O}}
\newcommand{\SSS}{\scriptscriptstyle}
\def\boldf#1{\hbox{\rlap{$#1$}\kern.4pt{$#1$}}}
\newcommand{\balpha}{{\boldf \alpha}}
\newcommand{\wh}{\hat w}
\newcommand{\Wh}{\hat W}
\newcommand{\wb}{\bar w}
\newcommand{\Wb}{\bar W}
\newcommand{\xb}{\bar x}
\newcommand{\cb}{\bar c}
\newcommand{\trans}{^{\scriptscriptstyle \top}}
\newcommand{\tW}{\tilde{W}}
\newcommand{\tw}{\tilde{w}}
\newcommand{\hbeta}{{\hat \beta}}
\newcommand{\De}{{\Delta}}

\newcommand{\figsheight}{4.0cm}
\renewcommand\baselinestretch{1}


\begin{titlepage}
\advance\topmargin by 0.5in
\begin{center}

{\Large Oracle Inequalities and Optimal Inference under Group
Sparsity}
\vspace{.8truecm}

\end{center}

\begin{center}

{\bf Karim Lounici}$^{(1,2)}$, {\bf  Massimiliano Pontil}$^{(3)}$, {\bf Alexandre B. Tsybakov}$^{(2)}$,
{\bf Sara van de Geer}$^{(4)}$
\vspace{.80truecm}

\noindent (1) Statistical Laboratory\\
University of Cambridge, Cambridge, UK\\
\vspace{.35truecm}

\noindent (2) CREST \\
3, Av. Pierre Larousse, 92240 Malakoff, France \\
{\em \{karim.lounici,alexandre.tsybakov\}@ensae.fr} \\

\vspace{.35truecm}

\noindent (3) Department of Computer Science \\
University College London \\
Gower Street, London WC1E, England, UK \\
E-mail: {\em m.pontil@cs.ucl.ac.uk}

\vspace{.35truecm}

\noindent (4) Seminar f\"ur Statistik, ETH Z\"urich \\
R\"amistrasse 101, 8092 Z\"urich, Switzerland \\
{\em geer@stat.math.ethz.ch}

\vspace{.65truecm}


\end{center}

\begin{abstract}
\noindent We consider the problem of estimating a sparse linear
regression vector $\beta^*$ under a gaussian noise model, for the
purpose of both prediction and model selection. We assume that prior
knowledge is available on the sparsity pattern, namely the set of
variables is partitioned into prescribed groups, only few of which
are relevant in the estimation process. This group sparsity
assumption suggests us to consider the Group Lasso method as a means
to estimate $\beta^*$. We establish oracle inequalities for the
prediction and $\ell_2$ estimation errors of this estimator. These
bounds hold under a restricted eigenvalue condition on the design
matrix. Under a stronger coherence condition, we derive bounds for
the estimation error for mixed $(2,p)$-norms with $1\le p\leq
\infty$. When $p=\infty$, this result implies that a threshold
version of the Group Lasso estimator selects the sparsity pattern of
$\beta^*$ with high probability. Next, we prove that the rate of
convergence of our upper bounds is optimal in a minimax sense, up to
a logarithmic factor, for all estimators over a class of group
sparse vectors. Furthermore, we establish lower bounds for the
prediction and $\ell_2$ estimation errors of the usual Lasso
estimator. Using this result, we demonstrate that the Group Lasso
can achieve an improvement in the prediction and estimation
properties as compared to the Lasso.

An important application of our results is provided by the problem
of estimating multiple regression equation simultaneously or
multi-task learning. In this case, our results lead to refinements of
the results in \cite{colt2009} and allow one to establish the
quantitative advantage of the Group Lasso over the usual Lasso in
the multi-task setting. Finally, within the same setting, we show
how our results can be extended to more general noise distributions,
of which we only require the fourth moment to be finite. To obtain
this extension, we establish a new maximal moment inequality, which
may be of independent interest.
\end{abstract}
\end{titlepage}

\section{Introduction}
%
Over the past few years there has been a great deal of attention on
the problem of estimating a {\em sparse}\footnote{The phrase
``$\beta^*$ is sparse'' means that most of the components of this
vector are equal to zero.} regression vector $\beta^*$ from a set of
linear measurements \beq y = X \beta^* + W. \label{eq:sr} \eeq Here
$X$ is a given $N \times K$ design matrix and $W$ is a zero mean
random variable modeling the presence of noise.

A main motivation behind sparse estimation comes from the
observation that in several practical applications the number of
variables $K$ is much larger than the number $N$ of observations,
but the underlying model is known to be sparse, see
\cite{candes2007dss,donoho2006srs} and references therein. In this
situation, the ordinary least squares estimator is not well-defined.
A more appropriate estimation method is the $\ell_1$-norm penalized
least squares method, which is commonly referred to as the Lasso.
The statistical properties of this estimator are now well
understood, see, e.g.,
\cite{BRT,BTWAnnals07,bunea2007soi,KoltStFlour,lounici2008snc,vandegeer2008hdg}
and references therein. In particular, it is possible to obtain
oracle inequalities on the estimation and prediction errors, which
are meaningful even in the regime $K \gg N$.

%
In this paper, we study the above estimation problem under
additional structural conditions on the sparsity pattern of the
regression vector $\beta^*$. Specifically, we assume that the set of
variables can be partitioned into
a number of groups, only few of which are relevant in the estimation
process. In other words, not only we require that many components of
the vector $\beta^*$ are zero, but also that many of a priori known
subsets of components are all equal to zero. This structured
sparsity assumption suggests us to consider the Group Lasso method
\cite{YuanLin} as a mean to estimate $\beta^*$ (see equation
\eqref{eq:opt} below). It is based on regularization with a mixed
$(2,1)$-norm, namely the sum, over the set of groups, of the square
norm of the regression coefficients restricted to each of the
groups. This estimator has received significant recent attention,
see
\cite{bach07,ChesHeb07,Horowitz08,zhang,koltch_y08,MeierGeerBuhlm06,MGB08,NarRin08,Obozinski2008usr,SPAM}
and references therein. Our principal goal is to clarify the
advantage of this more stringent group sparsity assumption in the
estimation process over the usual sparsity assumption. For this
purpose, we shall address the issues of bounding the prediction
error, the estimation error as well as estimating the sparsity
pattern. The main difference from most of the previous work is that
we obtain not only the upper bounds but also the corresponding lower
bounds and thus establish optimal rates of estimation and prediction
under group sparsity.

%
A main motivation for us to consider the group sparsity
assumption is the practically important problem of simultaneous
estimation the coefficient of multiple regression equations
\begin{equation}
\begin{array}{lcl}
y_1 & = & X_1 \beta^*_1 + W_1 \\
y_2 & = & X_2 \beta^*_2 + W_2 \\
~ & \vdots & ~\\
y_T & = & X_T \beta^*_T + W_T.
\end{array}
\label{eq:s1-i}
\end{equation}
Here $X_1,\dots,X_T$ are prescribed $n \times M$ design matrices,
$\beta_1^*,\dots,\beta^*_T \in \R^M$ are the unknown regression
vectors which we wish to estimate, $y_1\dots,y_T$ are
$n$-dimensional vectors of observations and $W_1,\dots,W_T$ are {\em
i.i.d.} zero mean random noise vectors. Examples in which this
estimation problem is relevant range from multi-task learning
\cite{AEP,maurer,Obozinski2008usr} and conjoint analysis
\cite{EPT,lenk} to longitudinal data analysis \cite{diggle2002ald}
as well as the analysis of panel data
\cite{hsiao2003apd,wooldridge2002eac}, among others. We briefly
review these different settings in the course of the paper. In
particular, multi-task learning provides a main motivation for our
study. In that setting each regression equation corresponds to a
different learning task; in addition to the requirement that $M \gg
n$, we also allow for the number of tasks $T$ to be much larger than
$n$. Following \cite{AEP} we assume that there are only few common
important variables which are shared by the tasks. That is, we
assume that the vectors $\beta_1^*,\dots,\beta^*_T$ are not only
sparse but also have their sparsity patterns included in the same
set of small cardinality. This group sparsity assumption
induces a relationship between the responses and, as we shall see,
can be used to improve estimation.

The model \eqref{eq:s1-i} can be reformulated as a single regression
problem of the form \eqref{eq:sr} by setting $K=MT$, $N=nT$,
identifying the vector $\beta$ by the concatenation of the vectors
$\beta_1,\ldots,\beta_T$ and choosing $X$ to be a block diagonal
matrix, whose blocks are formed by the matrices $X_1,\ldots,X_T$, in
order. In this way the above sparsity assumption on the vectors
$\beta_t$ translate in a group sparsity assumption on the vector
$\beta^*$, where each group is associated with one of the variables.
That is, each group contains the same regression component across
the different equations \eqref{eq:s1-i}. Hence the results developed
in this paper for the Group Lasso apply to the multi-task learning
problem as a special case.


\subsection{Outline of the main results}

We are now ready to summarize the main contributions of this paper.
\begin{itemize}

\item
We first establish bounds for the prediction and $\ell_2$ estimation
errors for the general Group Lasso setting, see Theorem \ref{th1}. In
particular, we include a ``slow rate" bound, which holds under no
assumption on the design matrix $X$. We then apply the theorem to
the specific multi-task setting, leading to some refinements of the
results in \cite{colt2009}. Specifically, we demonstrate that as the
number of tasks $T$ increases the dependence of the bound on the
number of variables $M$ disappears, provided that $M$ grows at the
rate slower than $\exp(T)$.

\item We extend previous results on the selection of the sparsity pattern
for the usual Lasso to the Group Lasso case, see Theorem
\ref{est-supnorm}. This analysis also allows us to establish the rates
of convergence of the estimators for mixed $(2,p)$-norms with $1\le
p\le \infty$ (cf. Corollary\ref{p-norm}).

\item We show that the rates of convergence in the above upper
bounds for the prediction and $(2,p)$-norm estimation errors are
optimal in a minimax sense (up to a logarithmic factor) for all
estimators over a class of group sparse vectors $\beta^*$, see
Theorem \ref{th_minimax_lb}.

\item We prove that the Group Lasso can achieve an
improvement in the prediction and estimation properties as compared
to the usual Lasso. For this purpose, we establish lower bounds for
the prediction and $\ell_2$ estimation errors of the Lasso estimator
(cf. Theorem \ref{thm:lb}) and show that, in some important cases,
they are greater than the corresponding upper bounds for the Group
Lasso, under the same model assumptions. In particular, we clarify
the advantage of the Group Lasso over the Lasso in the multi-task
learning setting.

\item Finally, we present an extension of the multi-task learning analysis to more
general noise distributions having only bounded fourth moment, see
Theorems \ref{thm:6.1} and \ref{thm:6.2}; this extension is not
straightforward and needs a new tool, the maximal moment inequality of
Lemma~\ref{nem-sara}, which may be of independent interest.

\end{itemize}

\subsection{Previous work}

Our results build upon recently developed ideas in the area of
compressed sensing and sparse estimation, see, e.g.,
\cite{BRT,candes2007dss,donoho2006srs,KoltStFlour} and references
therein. In particular, it has been shown by different authors,
under different conditions on the design matrix, that the Lasso
satisfies sparsity oracle inequalities, see
\cite{BRT,BTWAnnals07,bunea2007soi,lounici2008snc,KoltStFlour,
vandegeer2008hdg,ZH08} and references therein. Closest to our study
is the paper \cite{BRT}, which relies upon a Restricted Eigenvalue
(RE) assumption as well as \cite{lounici2008snc}, which considered
the problem of selection of sparsity pattern. Our techniques of
proofs build upon and extend those in these papers.

Several papers analyzing statistical properties of the Group Lasso
estimator appeared quite recently
\cite{bach07,ChesHeb07,Horowitz08,koltch_y08,MeierGeerBuhlm06,MGB08,NarRin08,SPAM}.
Most of them are focused on the Group Lasso for additive models
\cite{Horowitz08,koltch_y08,MGB08,SPAM} or generalized linear models
\cite{MeierGeerBuhlm06}. Special choice of groups is studied in
\cite{ChesHeb07}. Discussion of the Group Lasso in a relatively
general setting is given by Bach \cite{bach07} and Nardi and Rinaldo
\cite{NarRin08}. Bach \cite{bach07} assumes that the predictors
(rows of matrix $X$) are random with a positive definite covariance
matrix and proves results on consistent selection of sparsity
pattern $J(\beta^*)$ when the dimension of the model ($K$ in our
case) is fixed and $N\to\infty$. Nardi and Rinaldo \cite{NarRin08}
address the issue of sparsity oracle inequalities in the spirit of
\cite{BRT} under the simplifying assumption that all the Gram
matrices $\Psi_j$ (see the definition below) are proportional to the
identity matrix. However, the rates in their bounds are not precise
enough (see comments in \cite{colt2009}) and they do not demonstrate
advantages of the Group Lasso as compared to the usual Lasso.
Obozinski et al. \cite{Obozinski2008usr} consider the model
(\ref{eq:s1-i}) where all the matrices $X_t$ are the same and all
their rows are independent Gaussian random vectors with the same
covariance matrix. They show that the resulting estimator achieves
consistent selection of the sparsity pattern and that there may be
some improvement with respect to the usual Lasso. Note that the
Gaussian $X_t$ is a rather particular example, and Obozinski et al.
\cite{Obozinski2008usr} focused on the consistent selection, rather
than exploring whether there is some improvement in the prediction
and estimation properties as compared to the usual Lasso. The latter
issue has been addressed in our work \cite{colt2009} and in the
parallel work of Huang and Zhang \cite{zhang}. These papers
considered only heuristic comparisons of the two estimators, i.e.,
those based on the upper bounds. Also the settings treated there did
not cover the problem in whole generality. Huang and Zhang
\cite{zhang} considered the general Group Lasso setting but obtained
only bounds for prediction and $\ell_2$ estimation errors, while
\cite{colt2009} focused only on the multi-task setting, though
additionally with bounds for more general mixed $(2,p)$-norm
estimation errors and consistent pattern selection properties.

\subsection{Plan of the paper}
This paper is organized as follows. In Section \ref{sec:2} we define
the Group Lasso estimator and describe its application to the
multi-task learning problem.  In Sections \ref{sec:3} and
\ref{Section2.1} we study the oracle properties of this estimator in
the case of Gaussian noise, presenting upper bounds on the
prediction and estimation errors. In Section
\ref{sec:model-selection}, under a stronger condition on the design
matrices, we describe a simple modification of our method and show
that it selects the correct sparsity pattern with an overwhelming
probability. Next, in Section \ref{sec:minimax_lb} we show that the
rates of convergence in our upper bounds on prediction and
$(2,p)$-norm estimation errors with $1\le p\le \infty$ are optimal
in a minimax sense, up to a logarithmic factor. In Section
\ref{sec:lb} we provide a lower bound for the Lasso estimator, which
allows us to quantify the advantage of the Group Lasso over the
Lasso under the group sparsity assumption. In Section \ref{sec:5} we
discuss an extension of our results for multi-task learning to more
general noise distributions. Finally, Section \ref{sec:nem} presents
a new maximal moment inequality (an extension of Nemirovski's
inequality from the second to arbitrary moments), which is needed in
the proofs of Section \ref{sec:5}.

\section{Method}
\label{sec:2}


In this section, we introduce the notation and describe the
estimation method, which we analyze in the paper. We consider the
linear regression model \beq y = X \beta^{*} + W, \label{eq:model}
\eeq where $\beta^* \in \R^K$ is the vector of regression
coefficients, $X$ is an $N \times K$ design matrix, $y\in \mathbb
R^N$ is the response vector and $W \in \mathbb R^N$ is a random
noise vector which will be specified later. We also denote by
$x_{1}\trans,\dots,x_{N}\trans$ the rows of matrix $X$. Unless
otherwise specified, all vectors are meant to be column vectors.
Hereafter, for every positive integer $\ell$, we let $\N_\ell$ be
the set of integers from $1$ and up to $\ell$. Throughout the paper
we assume that $X$ is a deterministic matrix. However, it should be
noted that our results extend in a standard way (as discussed, e.g.,
in \cite{BRT}, \cite{candes2007dss}) to random $X$ satisfying the
assumptions stated below with high probability.

We choose $M \leq K$ and let the set ${G}_1,\dots,{G}_M$ form a
prescribed partition of the index set $\N_K$ in $M$ sets. That is,
$\N_K = \cup_{j=1}^M G_j$ and, for every $j \neq j'$, $G_j \cap G_{j'} = \emptyset$.
For every $j \in \N_M$, we let $K_j = |G_j|$ be the cardinality of
$G_j$ and denote by ${\bf X}_{G_j}$ the $N \times K_j$ sub-matrix of
$X$ formed by the columns indexed by $G_j$. We also use the notation
$\Psi = X\trans X/N$ and $\Psi_j = {\bf X}_{G_j}\trans {\bf
X}_{G_j}/N$ for the normalized Gram matrices of $X$ and ${\bf
X}_{G_j}$, respectively.

For every $\beta \in \R^{K}$ we introduce the notation $\beta^j =
(\beta_{k}: k \in {G}_j)$ and, for every $1\le p <\infty$, we define
the mixed $(2,p)$-norm of $\beta$ as $$\|\beta\|_{2,p} =
\left(\sum_{j=1}^M \left( \sum_{k \in G_j}
\beta_k^2\right)^\frac{p}{2} \right)^\frac{1}{p} =
\left(\sum_{j=1}^M
 \|\beta^j\|^p\right)^\frac{1}{p}
$$
and the $(2,\infty)$-norm of $\beta$ as
$$
\|\beta\|_{2,\infty} = \max_{1 \leq j \leq M}  \|\beta^j\|,
$$
where $\|\cdot\|$ is the standard Euclidean norm.

If $J \subseteq \N_M$ we let $\beta_{J}$ be the vector $(\beta^j
I\{j \in J\} : j \in \N_M)$, where $I\{\cdot\}$ denotes the
indicator function. Finally we set $J(\beta) = \{j:\beta^j \neq 0,~j
\in \N_M\}$ and $M(\beta) = |J(\beta)|$ where $|J|$ denotes the
cardinality of set $J\subset \{1,\dots,M\}$. The set $J(\beta)$
contains the indices of the relevant groups and the number $M(\beta)$
the number of such groups. Note that when $M=K$ we have $G_j=\{j\}$,
$j \in \N_K$ and $\|\beta\|_{2,p} = \|\beta\|_p$, where $\|\beta\|_p$
is the $\ell_p$ norm of $\beta$.

The main assumption we make on $\beta^*$ is that it is {\em
group sparse}, which means that $M(\beta^*)$ is much smaller
than $M$.

Our main goal is to estimate the vector $\beta^*$ as well as its
sparsity pattern $J(\beta^*)$ from $y$. To this end, we consider the
Group Lasso estimator. It is defined to be a solution $\hbeta$ of
the optimization problem \beq \min \left\{ \frac{1}{N} \|X \beta -
y\|^2 + 2 \sum_{j=1}^M \lam_j \|\beta^j\|: \beta \in \R^K \right\},
\label{eq:opt} \eeq where $\lam_1,\dots,\lam_M$ are positive
parameters, which we shall specify later.

In order to study the statistical properties of this estimator, it
is useful to present the optimality conditions for a solution of the
problem \eqref{eq:opt}. Since the objective function in
\eqref{eq:opt} is convex, $\hbeta$ is a solution of \eqref{eq:opt}
if and only if $0$ (the $K$-dimensional zero vector) belongs to the
subdifferential of the objective function. In turn, this condition
is equivalent to the requirement that $$ -\nabla \left(\frac{1}{N}
\|X \beta - y\|^2 \right) \in 2 \partial \left(\sum_{j=1}^M \lam_j
\|\hbeta^{j}\| \right),  $$ where $\partial$ denotes the
subdifferential (see, for example, \cite{BorLew} for more
information on convex analysis). Note that
\begin{align*}
\partial \left(\sum_{j=1}^M \lam_j \|\beta^{j}\| \right) =
\bigg\{ \theta \in \R^{K}: \theta^j = \lam_j
\frac{\beta^j}{\|\beta^j\|}~{\rm if}~ \beta^j \neq 0,~{\rm and}~
\|\theta^j\| \leq \lam_j~{\rm if}~ \beta^j = 0,~j \in \N_M
\bigg\}.
\end{align*}
Thus, $\hbeta$ is a solution of \eqref{eq:opt} if and only if
\begin{eqnarray}
\label{eq:sol-n0} \frac{1}{N} (X\trans (y - X \hbeta))^j & = &
\lam_j \frac{\hbeta^j}{\|\hbeta^j\|},~~~~\mbox{if}~\hbeta^j \neq 0 \\
\label{eq:sol-n2} \frac{1}{N} \|(X\trans (y - X \hbeta))^j\| & \leq
& \lam_j,~~~~~~~~~~~~~\mbox{if}~ \hbeta^j = 0.
\end{eqnarray}


\subsection{Simultaneous estimation of multiple regression equations and multi-task learning}
\label{sec:MT-model} As an application of the above ideas we
consider the problem of estimating multiple linear regression
equations simultaneously. More precisely, we consider multiple
Gaussian regression models,
\begin{equation}
\begin{array}{lcl}
y_1 & = & X_1 \beta^*_1 + W_1 \\
y_2 & = & X_2 \beta^*_2 + W_2 \\
~ & \vdots & ~\\
y_T & = & X_T \beta^*_T + W_T,
\end{array}
\label{eq:s1}
\end{equation}
where, for each $t \in \N_T$, we let $X_t$ be a prescribed $n \times
M$ design matrix, $\beta_t^*\in \R^M$ the unknown vector of
regression coefficients and $y_t$ an $n$-dimensional vector of
observations. We assume that $W_1,\dots,W_T$ are {\em i.i.d.} zero
mean random vectors.

We study this problem under the assumption that the sparsity
patterns of vectors $\beta_t^*$ are for any $t$ contained in the
same set of small cardinality $s$. In other words, the response
variable associated with each equation in \eqref{eq:s1} depends only
on some members of a small subset of the corresponding predictor
variables, which is preserved across the different equations. We
consider as our estimator a solution of the optimization problem
\beq
\min \left\{ \frac{1}{T} \sum_{t=1}^T \frac{1}{n} \|X_t \beta_t
- y_t\|^2 + 2 \lambda \sum_{j=1}^M \left(\sum_{t=1}^T
\beta_{tj}^2\right)^\frac{1}{2}: \beta_1,\dots,\beta_T \in \R^M
\right\} \label{eq:opt-MT} \eeq
with some tuning parameter
$\lambda>0$. As we have already mentioned in the introduction, this
estimator is an instance of the Group Lasso estimator described
above. Indeed, set $K = MT$, $N=nT$, let $\beta \in \R^K$ be the
vector obtained by stacking the vectors $\beta_1,\dots,\beta_T$ and
let $y$ and $W$ be the random vectors formed by stacking the vectors
$y_1,\dots,y_T$ and the vectors $W_1,\dots,W_T$, respectively.
We identify each row index of $X$ with a double index $(t,i) \in
\N_T \times \N_n$ and each column index with $(t,j) \in \N_T \times
\N_M$. In this special case the matrix $X$ is block diagonal and its
$t$-th block is formed by the $n \times M$ matrix $X_t$
corresponding to ``task $t$".
Moreover, the groups are defined as
$G_j = \{(t,j) : t \in \N_T\}$ and the parameters $\lam_j$ in
\eqref{eq:opt} are all set equal to a common value $\lam$. Within
this setting, we see that \eqref{eq:opt-MT} is a special case of
\eqref{eq:opt}.

Finally, note that the vectors $\beta^j = (\beta_{tj}: t \in
\N_T)^\top$ are formed by the coefficients corresponding to the
$j$-th variable ``across the tasks''. The set $J(\beta) =
\{j:\beta^j \neq 0,~j \in \N_M\}$ contains the indices of the
relevant variables present in at least one of the vectors
$\beta_1,\dots,\beta_T$ and the number $M(\beta)=|J(\beta)|$
quantifies the level of group sparsity across the tasks. The
structured sparsity (or group sparsity) assumption has the form
$M(\beta^*)\le s$ where $s$ is some integer much smaller than $M$.

\vspace{.3truecm} Our interest in this model with group sparsity is
mainly motivated by multi-task learning. Let us briefly discuss the
multi-task setting as well as other applications, in which the
problem of estimating multiple regression equations arises.

\vspace{.3truecm} \noindent {\bf Multi-task learning.} In machine
learning, the problem of multi-task learning has received much
attention recently, see \cite{AEP} and references therein. Here each
regression equation corresponds to a different ``learning task''. In
this context the tasks often correspond to binary classification,
namely the response variables are binary. For instance, in image
detection each task $t$ is associated with a particular type of
visual object (e.g., face, car, chair, etc.), the rows $x_{ti}^\top$
of the design matrix $X_t$ represent an image and $y_{ti}$ is a
binary label, which, say, takes the value $1$ if the image depicts
the object associated with task $t$ and the value $-1$ otherwise. In
this setting the number of samples $n$ is typically much smaller
than the number of tasks $T$. A main goal of multi-task learning is
to exploit possible relationships across the tasks to aid the learning
process.

\vspace{.3truecm} \noindent {\bf Conjoint analysis.} In marketing
research, an important problem is the analysis of datasets
concerning the ratings of different products by different customers,
with the purpose of improving products, see, for example,
\cite{aaker1995mr,lenk,EPT} and references therein. Here the index
$t\in \N_T$ refers to the customers and the index $i \in \N_n$
refers to the different ratings provided by a customer. Products are
represented by (possibly many) categorical or continuous variables
(e.g., size, brand, color, price etc.). The observation $y_{ti}$ is
the rating of product $x_{ti}$ by the $t$-th customer. A main goal
of conjoint analysis is to find common factors which determine
people's preferences to products. In this context, the variable
selection method we analyze in this paper may be useful to
``visualize'' peoples perception of products \cite{aaker1995mr}.

\vspace{.3truecm} \noindent {\bf Seemingly unrelated regressions
(SUR).} In econometrics, the problem of estimating the regression
vectors $\beta^*_t$ in \eqref{eq:s1} is often referred to as {\em
seemingly unrelated regressions} (SUR) \cite{zellner1962eme} (see
also \cite{srivastava1987sur} and references therein). In this
context, the index $i \in \N_n$ often refers to time and the
equations \eqref{eq:s1} are equivalently represented as $n$ systems
of linear equations, indexed by time. The underlying assumption in
the SUR model is that the matrices $X_t$ are of rank $M$, which
necessarily requires that $n \geq M$. Here we do not make such an
assumption. We cover the case $n\ll M$ and show how, under a
sparsity assumption, we can reliably estimate the regression
vectors. The classical SUR model assumes that the noise variables
are zero mean correlated Gaussian, with ${\rm cov}(W_s,W_t) =
\sigma_{st} I_{n \times n},~s,t \in \N_T$. This induces a relation
between the responses that can be used to improve estimation. In our
model such a relation also exists but it is described in a different
way, for example, we can consider that the sparsity patterns of
vectors $\beta_1^*,\dots,\beta^*_T$ are the same.

\vspace{.3truecm} \noindent {\bf Longitudinal and panel data.}
Another related context is {\em longitudinal data analysis}
\cite{diggle2002ald} as well as the analysis of {\em panel data}
\cite{hsiao2003apd,wooldridge2002eac}. Panel data refers to a
dataset which contains observations of different phenomena observed
over multiple instances of time (for example, election studies,
political economy data, etc). The models used to analyze panel data
appear to be related to the SUR model described above, but there is
a large variety of model assumptions on the structure of the
regression coefficients, see, for example,
\cite{hsiao2003apd}. Up to our knowledge however, sparsity assumptions have not
been been put forward for analysis within this context.


\section{Sparsity oracle inequalities}
\label{sec:3}

Let $1\le s \le M$ be an integer that gives an upper bound on the
group sparsity $M(\beta^*)$ of the true regression vector
$\beta^*$. We make the following assumption.
\begin{assumption}
\label{RE}
There exists a positive number $\kappa=\kappa(s)$ such that
\begin{align*}
\min \bigg\{ \frac{\|X\Delta\|}{\sqrt{N}\|\Delta_{J}\|}~:~ |J| \leq
s, \Delta\in \R^{K}\setminus \{0\}, \,\sum_{j \in J^c} \lam_j
\|\Delta^j\| \leq 3 \sum_{j \in J} \lam_j \|\Delta^j\| \bigg\}
\geq \kappa, \end{align*}\label{ass} where $J^c$ denotes the
complement of the set of indices~$J$.
\end{assumption}

To emphasize the dependency of Assumption \ref{RE} on $s$, we will
sometimes refer to it as Assumption RE($s$). This is a natural
extension to our setting of the Restricted Eigenvalue assumption for
the usual Lasso and Dantzig selector from \cite{BRT}. The $\ell_1$
norms are now replaced by (weighted) mixed (2,1)-norms.

Several simple sufficient conditions for Assumption \ref{RE} in the
Lasso case, i.e., when all the groups $G_j$ have size 1, are given
in \cite{BRT}. Similar sufficient conditions can be stated in our
more general setting. For example, Assumption 3.1 is immediately
satisfied if ${X\trans}X/N$ has a positive minimal eigenvalue. More
interestingly, it is enough to suppose that the matrix $X\trans X/N$
satisfies a Restricted Isometry condition as in \cite{candes2007dss}
or the coherence condition (cf. Lemma \ref{lem:2} below).

To state our first result we need some more notation. For every
symmetric and positive semi-definite matrix $A$, we denote by ${\rm
tr}(A)$, $\|A\|_{\rm Fr}$ and $\tn A \tn$ the trace, Frobenius and
spectral norms of $A$, respectively. If $\rho_1,\dots,\rho_k$ are the
eigenvalues of $A$, we have that ${\rm tr}(A) = \sum_{i=1}^k \rho_i$,
$\|A\|_{\rm Fr} = \sqrt{\sum_{i=1}^k \rho_i^2}$ and $\tn A \tn =
\max_{i=1}^k \rho_i$.
\begin{lemma}\label{lem:1}
Consider the model \eqref{eq:model}, and let $M \geq 2$, $N \geq 1$.
Assume that $W\in\R^N$ is a random vector with i.i.d.
$\GD(0,\sigma^2)$ gaussian components, $\sigma^2>0$.
For every $j \in \N_M$, recall that $\Psi_j={\bf X}_{G_j}\trans {\bf
X}_{G_j} /N$ and choose \beq \label{eq:lam} \lam_j \geq
\frac{2\sigma}{\sqrt{N}} \sqrt{{\rm tr}(\Psi_j)  + 2\tn \Psi_j\tn(2
q\log M + \sqrt{K_j q \log M})}. \eeq

Then with probability at least $1 - 2 M^{1-q}$, for any solution
$\hbeta$ of problem \eqref{eq:opt} and all $\beta \in \R^{K}$ we
have that
\begin{align}
\nonumber ~ &\frac{1}{N} \|X (\hbeta - \beta^*)\|^2 + \sum_{j=1}^M
\lam_j \|\hbeta^j- \beta^j\| \leq \frac{1}{N} \|X (\beta -
\beta^*)\|^2 &~\\ \label{eq:1} ~& \quad \quad \quad
+ 4 \sum_{j\in J(\beta)}\lam_j \min\left(\|\beta^j\|, \|\hbeta^{j}-\beta^j\|\right), &~\\
\label{eq:2}
~&\frac{1}{N} \|(X\trans X (\hbeta-\beta^*))^j\| \leq \frac{3}{2} \lam_j, & ~\\
\label{eq:3} ~ & M(\hbeta) \leq \frac{4 \phi_{\rm
max}}{\lam_{\min}^2 N} \|X (\hbeta - \beta^*)\|^2, & ~
\end{align}
where $\lam_{\min}= \min_{j=1}^M \lam_j$ and $\phi_{\rm max}$ is the
maximum eigenvalue of the matrix $ X\trans X/N$.
\end{lemma}

\begin{proof}
For all $\beta \in \R^{K}$, we have
$$
\frac{1}{N} \|X \hbeta - y\|^2 +  2 \sum_{j=1}^M \lam_j \|\hbeta^j\|
\leq \frac{1}{N} \|X \beta - y\|^2 +  2 \sum_{j=1}^M \lam_j
\|\beta^j\|,
$$
which, using $y=X\beta^*+W$, is equivalent to \beq \frac{1}{N} \|X
(\hbeta - \beta^*)\|^2 \leq \frac{1}{N} \|X(\beta - \beta^*)\|^2  +
\frac{2}{N} W\trans X (\hbeta-\beta)+ 2 \sum_{j=1}^M  \lam_j \big(
\|\beta^j\| - \|\hbeta^j\|\big). \label{eq:aa1} \eeq By the
Cauchy-Schwarz inequality, we have that
$$
W\trans X (\hbeta-\beta) \leq \sum_{j=1}^M \|(X\trans W)^j\|
\|\hbeta^j - \beta^j\|.
$$
For every $j \in \N_M$, consider the random event \beq \label{evtA}
\calA = \bigcap_{j=1}^M \calA_j,\eeq where \beq \label{evtAj} {\cal
A}_j = \left\{\frac1{N} \|(X\trans W)^j\| \leq \frac{\lam_j}{2}
\right\}. \eeq We note that
$$
\Prob\left({\cal A}_j\right) = \Prob\left(\left\{\frac{1}{N^2}
W\trans {\bf X}_{G_j} {\bf X}_{G_j} \trans W \leq \frac{\lam_j^2}{4}
\right\}\right) =
\Prob\left(\left\{\frac{\sum_{i=1}^N \eigv_{j,i} (\xi_i^2 -
1)}{\sqrt{2} \|\eigv_j\|} \leq x_j \right\}\right),
$$ where
$\xi_1,\dots,\xi_N$ are i.i.d. standard Gaussian,
$\eigv_{j,1},\dots,\eigv_{j,N}$ denote the eigenvalues of the matrix
${\bf X}_{G_j} {\bf X}_{G_j}\trans/N$, among which the positive ones
are the same as those of $\Psi_j$, and the quantity $x_j$ is defined
as
$$
x_j = \frac{\lam_j^2 N/(4\sigma^2)- {\rm tr}(\Psi_j)}{\sqrt{2}
\|\Psi_j\|_{\rm Fr}}.
$$
We apply Lemma \ref{chi2} to upper bound the probability of the
complement of the event $A_j$. Specifically, we choose $v =
(v_{j,1},\dots,v_{j,N})$, $x = x_j$ and $m(v) =
\tn \Psi_j \tn/\|\Psi_j\|_{\rm Fr}$ and conclude from Lemma \ref{chi2}
that
\begin{eqnarray*}
\Prob\big({\cal A}_j^c\big)
\leq 2 \exp\left( -\frac{x_j^2} {2(1 + \sqrt{2} x_j
\tn \Psi_j\tn/\|\Psi_j\|_{\rm Fr})}\right).
\end{eqnarray*}
We now choose $x_j$ so that the right hand side of the above
inequality is smaller than $2 M^{-q}$. A direct computation yields
that $$ x_j
\geq \sqrt{2} \tn \Psi_j\tn/\|\Psi\|_{\rm Fr} q \log M + \sqrt{2
(\tn \Psi_j\tn q \log M)^2 + 2 q \log M},
$$
which, using the subadditivity property of the square root and the
inequality $\|\Psi_j\|_{\rm Fr} \leq \sqrt{K_j} \tn\Psi_j\tn$ gives
inequality \eqref{eq:lam}.
We conclude, by a union bound, under the above condition on the
parameters $\lam_j$, that $\Prob ({\cal A}^c) \leq 2 M^{1-q}$. Then,
it follows from inequality \eqref{eq:aa1}, with probability at least
$1 - 2 M^{1-q}$, that
\begin{eqnarray}
\nonumber \frac{1}{N} \|X (\hbeta - \beta^*)\|^2 + \sum_{j=1}^M
\lam_j\|\hbeta^j-\beta^j\| & \leq & \frac{1}{N} \|X (\beta -
\beta^*)\|^2 + 2 \sum_{j=1}^M\lam_j\big(
\|\hbeta^j-\beta^j\| + \|\beta^j\| - \|\hbeta^j\|\big) \\
\nonumber ~ & \leq & \frac{1}{N} \|X (\beta - \beta^*)\|^2 + 4
\sum_{j \in J(\beta)} \lam_j \min\left(\|\beta^j\|,
\|\hbeta^{j}-\beta^j\|\right),
\end{eqnarray}
which coincides with inequality \eqref{eq:1}.

To prove \eqref{eq:2}, we use the
inequality \beq \frac{1}{N} \|(X\trans(y-X\hbeta))^j\| \leq \lam_j,
\label{eq:nec} \eeq which follows from the optimality conditions
(\ref{eq:sol-n0}) and (\ref{eq:sol-n2}). Moreover, using
equation \eqref{eq:model} and the triangle inequality, we obtain that
\begin{align}
\nonumber &\frac{1}{N} \|(X\trans X(\hbeta - \beta^*))^j\| \leq
\frac{1}{N} \|(X\trans(X\hbeta - y))^j\| + \frac{1}{N} \|(X\trans
W)^j\|. \nonumber \end{align} The result then follows by combining
the last inequality with inequality \eqref{eq:nec} and using the
definition of the event ${\cal A}$.

Finally, we prove \eqref{eq:3}. First, observe that, on the event
${\cal A}$, it holds, uniformly over $j \in \N_M$, that \beq
\nonumber \frac{1}{N} \|(X\trans X(\hbeta - \beta^*))^j\| \geq
\frac{\lam_j}{2},~~~{\rm if}~\hbeta^j \neq 0. \eeq This fact follows
from \eqref{eq:sol-n0}, \eqref{eq:model} and the definition of the
event ${\cal A}$. The following chain yields the result:
\begin{eqnarray}
\nonumber M(\hbeta) & \leq & \frac{4}{N^2} \sum_{j\in J(\hbeta)}
\frac{1}{\lam_j^2} \|(X\trans X (\hbeta - \beta^*))^j\|^2 \\
\nonumber ~ & \leq &\frac{4}{\lam_{\min}^2 N^2}
\sum_{j \in J(\hbeta)} \|(X\trans X (\hbeta - \beta^*))^j\|^2 \\
\nonumber ~ & \leq & \frac{4}{\lam_{\min}^2 N^2}
\|X\trans X (\hbeta - \beta^*)\|^2 \nonumber \\
& \leq & \frac{4\phi_{\rm max}} {\lam_{\min}^2 N} \|X(\hbeta-
\beta^*)\|^2, \nonumber
\end{eqnarray}
where, in the last line we have used the fact that the eigenvalues
of $X\trans X/N$ are bounded from above by $\phi_{\max}$.
\end{proof}
We are now ready to state the main result of this section.

\begin{theorem}\label{th1}
Consider the model \eqref{eq:model} and let $M \geq 2$, $N \geq 1$.
Assume that $W\in\R^N$ is a random vector with i.i.d.
$\GD(0,\sigma^2)$ gaussian components, $\sigma^2>0$. For every $j
\in \N_M$, define the matrix $\Psi_j={\bf X}_{G_j}\trans {\bf
X}_{G_j} /N$ and choose
$$
\lam_j \geq \frac{2\sigma}{\sqrt{N}} \sqrt{{\rm tr}(\Psi_j)  +
2\tn \Psi_j \tn(2 q\log M + \sqrt{K_j q \log M})}.
$$
Then with probability at least $1 - 2 M^{1-q}$, for any solution
$\hbeta$ of problem \eqref{eq:opt} we have that
\begin{eqnarray}
\frac{1}{N} \|X(\hbeta - \beta^*)\|^2 & \leq & 4\|\beta^*\|_{2,1}\,
\max_{j=1}^M \lam_j.  \label{eq:t0-GL}
\end{eqnarray}
If, in addition, $M(\beta^{*})\leq s$ and Assumption \ref{ass} holds
with $\kappa=\kappa(s)$, then with probability at least $1 - 2
M^{1-q}$, for any solution $\hbeta$ of problem \eqref{eq:opt} we
have that
\begin{eqnarray}
\frac{1}{N} \|X(\hbeta - \beta^*)\|^2 & \leq & \frac{16}{\kappa^2}
\sum_{j \in J(\beta^*)} \lam_j^2,
\label{eq:t1-GL}
\\
\label{eq:t2-GL} \|\hbeta - \beta^*\|_{2,1} & \leq &
\frac{16}{\kappa^2} \sum_{j \in J(\beta^*)}
\frac{\lam_j^2}{\lam_{\min}},
~~ \\
\label{eq:t3-GL} M(\hbeta) & \leq & \frac{64 \phi_{\rm
max}}{\kappa^2} \sum_{j \in J(\beta^*)}
\frac{\lam^2_j}{\lam^2_{\min}},~~
\end{eqnarray}
where $\lam_{\min} = \min_{j=1}^M \lam_j$ and $\phi_{\rm max}$ is
the maximum eigenvalue of the matrix $ X\trans X/N$. If, in
addition, Assumption RE(2$s$) holds, then with the same probability
for any solution $\hbeta$ of problem \eqref{eq:opt} we have that
\begin{eqnarray}
\label{eq:t4-GL} \|\hbeta - \beta^*\| & \leq & \frac{4
\sqrt{10}}{\kappa^2(2s)} \frac{\sum_{j \in J(\beta^*)}
\lam_j^2}{\lam_{\min}\sqrt{s}}.
\end{eqnarray}
\end{theorem}

\begin{proof} Inequality (\ref{eq:t0-GL}) follows immediately from
(\ref{eq:1}) with $\beta=\beta^*$. We now prove the remaining
assertions. Let $J = J(\beta^*) = \{j: (\beta^*)^j \neq 0\}$ and let
$\De=\hbeta-\beta^*$. By inequality \eqref{eq:1} with $\beta =
\beta^*$ we have, on the event~${\cal A}$, that \beq \frac{1}{N} \|X
\De\|^2 \leq 4 \sum_{j \in J} \lam_j \|\De^j\| \leq 4 \sqrt{\sum_{j
\in J} \lam_j^2} \,\| \De_{J}\|. \label{eq:gino} \eeq Moreover by
the same inequality, on the event ${\cal A}$, we have that $ \sum_{j=1}^M
\lam_j \|\De^j\| \leq 4 \sum_{j \in J} \lam_j \|\De^j\|$, which
implies that $\sum_{j \in J^c} \lam_j \|\De^j\| \leq 3  \sum_{j \in
J} \lam_j \|\De^j\|$. Thus, by Assumption \ref{ass} \beq \|\De_{J}\|
\leq \frac{\| X \De\|}{\kappa \sqrt{N}}. \label{eq:bb} \eeq Now,
\eqref{eq:t1-GL} follows from \eqref{eq:gino} and \eqref{eq:bb}.

Inequality \eqref{eq:t2-GL} follows by noting that, by \eqref{eq:1},
$$
\sum_{j=1}^M \lam_j \|\De^j\| \leq 4 \sum_{j \in J} \lam_j \|\De^j\|
\leq 4 \sqrt{\sum_{j \in J} \lam_j^2} \|\De_{J}\| \leq 4
\sqrt{\sum_{j \in J} \lam_j^2} \frac{\|X \De \|}{\sqrt{N}\kappa}
$$
and then using \eqref{eq:t1-GL} and $\sum_{j=1}^M \|\De^j\| \leq
\sum_{j=1}^M  \|\De^j\|\lam_j/\lam_{\min}$.

Inequality
\eqref{eq:t3-GL} follows from \eqref{eq:3} and \eqref{eq:t1-GL}.

Finally, we prove \eqref{eq:t4-GL}. Let $J'$ be the set of indices
in $J^c$ corresponding to $s$ largest values of $\lam_j
\|\Delta^{j}\|$. Consider the set $J_{2s}=J\cup J'$. Note that
$|J_{2s}|\leq 2s$. Let $j(k)$ be the index of the $k-$th largest
element of the set $\{\lam_j \|\Delta^{j}\|: \,j\in J^c\}$. Then,
$$
\lam_{j(k)} \|\Delta^{j(k)}\| \le \sum_{j\in J^c}\lam_j
\|\Delta^{j}\|/k.
$$
This and the fact that $\sum_{j\in J^c} \lam_j \|\Delta^j\|\le 3
\sum_{j \in J} \lam_j \|\Delta^{j}\|$ on the event ${\cal A}$
implies
\begin{eqnarray*}
\sum_{j\in J_{2s}^c} \lam_j^2 \|\Delta^j\|^2 &\le&
\sum_{k=s+1}^\infty \frac{\left(\sum_{\ell \in J^c}
\lam_{\ell}\|\Delta^{\ell}\|\right)^2
}{k^2} \\
&\le& \frac{ \left(\sum_{\ell \in J^c}
\lam_{\ell}\|\Delta^{\ell}\|\right)^2 }{s}\le \frac{9
\left(\sum_{\ell \in J} \lam_{\ell}\|\Delta^{\ell}\|\right)^2
}{s}\\
&\le& \frac{9(\sum_{j\in J}\lam_j^2) \|\Delta_{J}\|^2}{s} \le
\frac{9(\sum_{j\in J}\lam_j^2) \|\Delta_{J_{2s}}\|^2}{s}.
\end{eqnarray*}
Therefore, it follows that
$$
\lam_{\min}^2 \|\Delta_{J_{2s}^c}\|^2
\leq \frac{9}{s} \sum_{j \in J} \lam_j^2 \|\Delta_{J_{2s}}\|^2
$$
and, in turn, that
\begin{eqnarray}\label{ona}
\|\Delta\|^2 \le \frac{10}{s} \sum_{j \in J}
\frac{{\lam_j}^2}{\lam_{\min}^2} \|\Delta_{J_{2s}}\|^2.
\end{eqnarray}
Next note from \eqref{eq:gino} that
\begin{eqnarray}
\hspace{-.2truecm}\frac{1}{N} \|X \Delta\|^2
 \leq 4 \sqrt{\sum_{j \in J} \lam_j^2}
 \|\Delta_{J_{2s}}\|.\label{eq:gino1}
\end{eqnarray}
In addition, $ \sum_{j \in J^c} \lam_j \|\Delta^j\|\le 3 \sum_{j \in
J} \lam_j \|\Delta^j\| $ easily implies that
$$
\sum_{j \in J_{2s}^c} \lam_j \|\Delta^j\|\le 3 \sum_{j \in J_{2s}}
\lam_j \|\Delta^j\|.
$$
Combining Assumption RE(2$s$) with \eqref{eq:gino1} we have, on the
event ${\cal A}$, that
$$
\|\Delta_{J_{2s}}\| \le \frac{4 \sqrt{\sum_{j \in J} \lam_j^2}}
{\kappa^2(2s)}.
$$
This inequality and \eqref{ona} yield \eqref{eq:t4-GL}. \end{proof}

The oracle inequality (\ref{eq:t1-GL}) of Theorem \ref{th1} can be
generalized to include the bias term as follows.

\begin{theorem}\label{th-bias}
Let the assumptions of Lemma~\ref{lem:1} be satisfied and let
Assumption \ref{ass} holds with
$\kappa=\kappa(s)$ and with factor 3 replaced by 7. Then with
probability at least $1 - 2 M^{1-q}$, for any solution $\hbeta$ of
problem \eqref{eq:opt} we have
$$
\frac{1}{N}\|X(\hat{\beta}- \beta^*)\|^2 \leq \min\left\{
\frac{96}{\kappa^2}\sum_{j\in J(\beta)}\lambda_j^2 +
\frac{2}{N}\|X(\beta - \beta^*)\|^2: \beta \in \R^K, M(\beta)\leq s
 \right\}.
$$
\end{theorem}
This result is of interest when $\beta^*$ is only assumed to
approximately sparse, that is when there exists a set of indices
$J_0$ with cardinality smaller than $s$ such that
$\|(\beta^*)_{J_0^c}\|^2$ is small.

\begin{proof}
  Let $\beta$ be arbitrary. Set $\Delta = \hat{\beta}-\beta$. By
  inequality (\ref{eq:1}), we have, on the event $\mathcal{A}$ that
  $$
\frac{1}{N}\|X(\hat{\beta}-\beta^*)\|^2 + \sum_{j=1}^M\lambda_j
\|\Delta^j\| \leq \frac{1}{N}\| X(\beta -\beta^*) \|^2 + 4\sum_{j\in
J(\beta)}\lambda_j \|\Delta^j\|.
  $$
Let $y>0$ be arbitrary. We consider two cases:\medskip\\
\medskip
case i) $4 \sum_{j\in J(\beta)}\lambda_j \|\Delta^j\| \geq
\frac{1}{N}\|X(\beta - \beta^*)\|^2$\\
case ii) $4 \sum_{j\in J(\beta)}\lambda_j \|\Delta^j\| <
\frac{1}{N}\|X(\beta - \beta^*)\|^2$\\

In case i), we have
$$
\frac{1}{N}\| X(\hat{\beta}-\beta^*) \|^2 + \sum_{j=1}^M \lambda_j
\|\Delta^j\| \leq  8 \sum_{j\in J(\beta)}\lambda_j \|\Delta^j\|.
$$
This implies
$$
\sum_{j\in J(\beta)^c}\lambda_j\|\Delta^j\| < 7 \sum_{j\in
J(\beta)}\lambda_j\|\Delta^j\|.
$$
Thus, by Assumption \ref{ass} (with factor 3 replaced by 7), we have
$$
\|\Delta_{J(\beta)}\| \leq \frac{\|X\Delta\|}{\kappa \sqrt{N}}.
$$
We obtain
\begin{eqnarray*}
  \frac{1}{N}\|X(\hat{\beta} - \beta^*)\|^2 + \sum_{j=1}^M \lambda_j
  \|\Delta^j\| &\leq& \frac{8}{\kappa}\sqrt{\sum_{j\in
  J(\beta)}\lambda_j^2}\frac{\|X\Delta\|}{\sqrt{N}}\\
  &\leq& \frac{8}{\kappa}\sqrt{\sum_{j\in  J(\beta)}\lambda_j^2}\left[ \frac{\|X(\hat{\beta} - \beta^*) \|}{\sqrt{N}} + \frac{\| X(\beta - \beta^*) \|}{\sqrt{N}}
  \right]\\
  &\leq& \frac{1}{2}\frac{\|X(\hat{\beta}-\beta^*)\|^2}{N} +
  \frac{32}{\kappa^2}\sum_{j\in J(\beta)}\lambda_j^2\\
  &~&\hspace{3cm} + \frac{\|X(\beta -
  \beta^*)\|^2}{N} + \frac{16}{\kappa^2}\sum_{j\in J(\beta)}\lambda_j^2.
\end{eqnarray*}
Hence
$$
\frac{1}{N}\|X(\hat{\beta} - \beta^*)\|^2 + 2\sum_{j=1}^M
\lambda_j\|\Delta^j\| \leq \frac{96}{\kappa^2}\sum_{j\in
J(\beta)}\lambda_j^2 + \frac{2}{N}\|X(\beta - \beta^*)\|^2.
$$
Case ii) gives
$$
\frac{1}{N}\|X(\hat{\beta} - \beta^*)\|^2 + \sum_{j=1}^M\lambda_j
\|\Delta^j\| < \frac{2}{N}\|X(\beta - \beta^*)\|^2.
$$
Hence
$$
\frac{1}{N}\|X(\hat{\beta}- \beta^*)\|^2 \leq
\min_\beta\left[\frac{96}{\kappa^2}\sum_{j\in J(\beta)}\lambda_j^2 +
\frac{2}{N}\|X(\beta - \beta^*)\|^2
  \right].
$$

\end{proof}


We end this section by a remark about the Group Lasso estimator with
overlapping groups, i.e., when $\N_K = \cup_{j=1}^M G_j$ but $G_j \cap G_{j'}\neq \emptyset$ for some
$j,j' \in \N_M$, $j \neq j'$. We refer
to \cite{binyu} for motivation and discussion featuring the
statistical relevance of group sparsity with overlapping
groups. Inspection of the proofs of Lemma \ref{lem:1} and Theorem
\ref{th1} immediately yields the following conclusion.
\begin{remark}
Inequalities \eqref{eq:1} and \eqref{eq:2} in Lemma \ref{lem:1} and
inequalities \eqref{eq:t1-GL}--\eqref{eq:t3-GL} in Theorem \ref{th1}
remain correct in the more general case of overlapping groups
$G_1,\dots, G_M$.
\end{remark}


\section{Sparsity oracle inequalities for
multi-task learning} \label{Section2.1}

We now apply the above results to the multi-task learning problem
described in Section \ref{sec:MT-model}. In this setting, $K=MT$ and
$N=nT$, where $T$ is the number of tasks, $n$ is the sample size for
each task and $M$ is the nominal dimension of unknown regression
parameters for each task. Also, for every $j \in \N_M$, $K_j = T$
and $\Psi_j= (1/T) I_{T \times T}$, where $I_{T \times T}$ is the $T
\times T$ identity matrix.
This fact is a consequence of the block diagonal structure of the design
matrix $X$ and the assumption that the variables are normalized to
one, namely all the diagonal elements of the matrix $(1/n) X\trans
X$ are equal to one. It follows that ${\rm tr}(\Psi_j) = 1$ and
$\tn \Psi_j\tn= 1/T$. The regularization parameters $\lam_j$ are all
equal to the same value $\lam$, cf. (\ref{eq:opt-MT}). Therefore,
\eqref{eq:lam} takes the form
\begin{equation}\label{q1}
\lam \geq \frac{2\sigma}{\sqrt{nT}} \sqrt{1 + \frac{2}{T} \left(2q
\log M + \sqrt{Tq \log M}\right)}.
\end{equation}
In particular, Lemma \ref{lem:1} and Theorem \ref{th1} are valid for
$$
\lam \geq \frac{2\sqrt{2}\sigma}{\sqrt{nT}} \sqrt{1 + \frac{5q}{2}
\frac{\log M}{T}}
$$
since the right-hand side of this inequality is greater than that of
(\ref{q1}).

 For the convenience of the reader we state the
Restricted Eigenvalue assumption for the multi-task case
\cite{colt2009}.
\begin{assumption}
\label{RE-MT}
There exists a positive number $\kappa_{\rm MT}=\kappa_{\rm MT}(s)$
such that
\begin{align*}
\min \bigg\{ \frac{\|X\Delta\|}{\sqrt{n}\|\Delta_{J}\|}~:~ |J| \leq
s, \Delta\in \R^{MT}\setminus \{0\}, \|\Delta_{J^c}\|_{2,1} \leq 3
\|\Delta_{J}\|_{2,1} \bigg\} \geq \kappa_{\rm MT},
\end{align*}\label{ass_MT} where $J^c$ denotes the complement of the
set of indices~$J$.
\end{assumption}
We note that parameters $\kappa,\phi_{\max }$ defined in Section
\ref{sec:3} correspond to $\kappa_{\rm MT} /\sqrt{T}$ and $\phi_{\rm
MT}/T$ respectively, where $\phi_{\rm MT}$ is the largest eigenvalue
of the matrix $X\trans X/n$.

Using the above observations we obtain the following corollary of
Theorem \ref{th1}.

\begin{corollary}\label{co1}
Consider the multi-task model \eqref{eq:s1} for $M \geq 2$ and $T,n
\geq 1$. Assume that $W\in\R^N$ is a random vector with i.i.d.
$\GD(0,\sigma^2)$ gaussian components, $\sigma^2>0$, and all diagonal elements of the matrix $X\trans X/n$ are equal to $1$. Set
$$
\lambda = \frac{2\sqrt{2}\sigma}{\sqrt{nT}}\left(1 + \frac{A \log
M}{T}\right)^{1/2},
$$
where $A>5/2$. Then with probability at least $1 - 2M^{1-2A/5}$, for
any solution $\hbeta$ of problem \eqref{eq:opt-MT} we have that
\begin{eqnarray}
\label{eq:t0} \frac{1}{nT} \|X(\hbeta - \beta^*)\|^2
\hspace{-.2truecm} & \leq & \frac{8\sqrt{2} \sigma}{\sqrt{nT}}
\left(1 + \frac{A \log M}{T}\right)^{1/2}\|\beta^*\|_{2,1}\,.
\end{eqnarray}
Moreover, if in addition it holds that $M(\beta^*) \leq s$ and Assumption \ref{RE-MT}
holds with $\kappa_{\rm MT}=\kappa_{\rm MT}(s)$, then with
probability at least $1 - 2M^{1-2A/5}$, for any solution $\hbeta$ of
problem \eqref{eq:opt-MT} we have that
\begin{eqnarray}
\label{eq:t1} \hspace{-.6truecm} \frac{1}{nT} \|X(\hbeta -
\beta^*)\|^2 \hspace{-.2truecm} & \leq & \hspace{-.2truecm}
\frac{128 \sigma^2}{\kappa_{\rm MT}^2} \frac{s}{n}\left(1 + \frac{A
\log
M}{T}\right)~~\hspace{.3truecm} \\
\label{eq:t2} \frac1{\sqrt {T}}\|\hbeta - \beta^*\|_{2,1}
\hspace{-.2truecm} & \leq & \hspace{-.2truecm} \frac{32\sqrt{2}
\sigma}{\kappa_{\rm MT}^2} \frac{s}{\sqrt{n}}\left(1 + \frac{A \log
M}{T}\right)^{1/2}~~ \\
\label{eq:t3} \hspace{-.2truecm} M(\hbeta) \hspace{-.2truecm} & \leq
& \hspace{-.2truecm}  \frac{64 \phi_{\rm MT}}{\kappa_{\rm MT}^2}
s,~~
\end{eqnarray}
where $\phi_{\rm MT}$ is the largest eigenvalue of the matrix
$X\trans X/n$.

Finally, if in addition $\kappa_{\rm MT}(2s) > 0$, then with
the same probability for any solution $\hbeta$ of problem
\eqref{eq:opt-MT} we have that
\begin{eqnarray}
\label{eq:t4} \hspace{-.2truecm} \frac{1}{\sqrt{T}} \|\hbeta -
\beta^*\| \hspace{-.2truecm} & \leq & \hspace{-.2truecm} \frac{16
\sqrt{5} \sigma}{\kappa^2_{\rm MT}(2s)} \sqrt{\frac{s}{n}}\left(1 +
\frac{A \log M}{T}\right)^{1/2}. \hspace{.3truecm}
\end{eqnarray}
\end{corollary}

Note that the values $T$ and $\sqrt{T}$ in the denominators of the
left-hand sides of inequalities (\ref{eq:t1}), (\ref{eq:t2}), and
(\ref{eq:t4}) appear quite naturally. For instance, the norm
$\|\hbeta - \beta^*\|_{2,1}$ in (\ref{eq:t2}) is a sum of $M$ terms
each of which is a Euclidean norm of a vector in $\R^T$, and thus it
is of the order $\sqrt{T}$ if all the components are equal.
Therefore, (\ref{eq:t2}) can be interpreted as a correctly
normalized ``error per coefficient" bound.

Corollary \ref{co1} is valid for any fixed $n,M,T$; the approach is
non-asymptotic. Some relations between these parameters are relevant
in the particular applications and various asymptotics can be
derived as special cases. For example, in multi-task learning it is
natural to assume that $T \geq n$, and the motivation for our
approach is the strongest if also $M\gg n$. The bounds of Corollary
\ref{co1} are meaningful if the sparsity index $s$ is small as
compared to the sample size $n$ and the logarithm of the dimension
$\log M$ is not too large as compared to ${T}$.

More interestingly, the dependency on the dimension $M$ in the
bounds is negligible if the number of tasks $T$ is larger than $\log
M$. In this regime, no relation between the sample size $n$ and the
dimension $M$ is required. This is quite in contrast to the standard
results on sparse recovery where the condition
\begin{eqnarray*} \log ({\rm
dimension})&\ll& {\rm sample \ size}
\end{eqnarray*}
 is considered as {\it sine qua
non} constraint. For example, Corollary \ref{co1} gives meaningful
bounds if $M=\exp({n^\gamma})$ for arbitrarily large $\gamma>0$,
provided that $T>n^{\gamma}$.

Finally, note that Corollary \ref{co1} is in the same spirit as a
result that we obtained in \cite{colt2009} but there are two
important differences. First, in \cite{colt2009} we considered
larger values of $\lambda$, namely with $\left(1 + \frac{A \log
M}{\sqrt{T}}\right)^{1/2}$ in place of $\left(1 + \frac{A \log
M}{T}\right)^{1/2}$, and we obtained a result with higher
probability. We switch here to the smaller $\lambda$ since it leads
to minimax rate optimality, cf. lower bounds below. The second
difference is that we include now the ``slow rate" result
(\ref{eq:t0}), which guarantees convergence of the prediction loss
{\it with no restriction on the matrix} $X\trans X$, provided that
the norm $(2,1)$-norm of $\beta^*$ is bounded. For example, if the
absolute values of all components of $\beta^*$ do not exceed some
constant $\beta_{\max}$, then $\|\beta^*\|_{2,1}\le \beta_{\max}
s\sqrt{T}$ and the bound (\ref{eq:t0}) is of the order
$\frac{s}{\sqrt{n}}\left(1 + \frac{A \log M}{T}\right)^{1/2}$.


\section{Coordinate-wise estimation and selection of sparsity pattern}
\label{sec:model-selection}

In this section we show how from any solution of (\ref{eq:opt}), we
can estimate the correct sparsity pattern $J(\beta^*)$ with
high probability. We also establish bounds for estimation of
$\beta^*$ in all $(2,p)$ norms with $1\le p\le \infty$ under a
stronger condition than Assumption \ref{RE}.

Recall that we use the notation $\Psi = \frac{1}{N}X\trans X$ for
the Gram matrix of the design. We introduce some additional notation
which will be used throughout this section. For any $j,j'$ in
$\mathbb N_{M}$ we define the matrix $\Psi[j,j'] = \frac{1}{N}{\bf
X}_{G_j}\trans {\bf X}_{G_{j'}}$ (note that $\Psi[j,j]= \Psi_j$ for
any $j$). We denote by $\Psi[j,j']_{t,t'}$, where $t \in \N_{K_j},t'
\in \N_{K_{j'}}$, the $(t,t')$-th element of matrix $\Psi[j,j']$.
For any $\Delta\in \mathbb R^K$ and $j\in \mathbb N_M$ we set
$\Delta^j = (\Delta_{t}: t \in \N_{K_j})$.

In this section, we assume that the following condition holds true.
\begin{assumption}\label{MC}
There exist some integer $s\ge 1$ and some constant $\alpha>0$ such
that:
\begin{enumerate}
\item For any $j\in \mathbb N_M$ and $t \in \N_{K_j}$ it holds that $(\Psi[j,j])_{t,t} = \phi$ and
$$
\max_{1\leq t,t'\leq K_j,t\neq t'}\left|
(\Psi[j,j])_{t,t'} \right|
\leq
\frac{\lambda_{\min}\phi}{14\alpha\lambda_{\max}s}\frac{1}{\sqrt{K_j
K_{j'}}}.
$$
\item For any $j\neq j' \in \mathbb N_M$ it holds that
$$
\max_{1\leq t \leq \min(K_j,K_{j'})}\left| (\Psi[j,j'])_{t,t}
\right| \leq \frac{\lambda_{\min}\phi}{14\alpha\lambda_{\max}s}
$$
and
$$
\max_{1\leq t \leq K_j, 1\leq t' \leq K_{j'},t\neq
t'}\left| (\Psi[j,j'])_{t,t'} \right| \leq
\frac{\lambda_{\min}\phi}{14\alpha\lambda_{\max}s}\frac{1}{\sqrt{K_j
K_{j'}}}.
$$
\end{enumerate}
\end{assumption}
This assumption is an extension to the general Group Lasso setting
of the coherence condition of \cite{colt2009} introduced in the
particular multi-task setting. Indeed, in the multi-task case
$K_j\equiv T$, $\lambda_{\min}=\lambda_{\max}$, and for any $j\in
\mathbb{N}_M$ the matrix ${\bf X}_{G_j}$ is block diagonal with the
$t$-th block of size $n\times 1$ formed by the $j$-th column of the
matrix $X_{t}$ (recall the notation in Section \ref{sec:MT-model})
and $\phi = 1/T$. It follows that $(\Psi[j,j'])_{t,t'} = 0$ for any
$j,j' \in \mathbb{N}_M$ and $t\neq t'\in \mathbb N_{T}$. Then
Assumption \ref{MC} reduces to the following: $\max_{1 \leq t \leq
T}|(\Psi[j,j'])_{t,t}|\leq \frac{1}{14\alpha sT}$ whenever $j\neq
j'$ and $(\Psi[j,j])_{t,t} = \frac{1}{T}$. Thus, we see that for the
multi-task model Assumption \ref{MC} takes the form of the usual
coherence assumption for each of the $T$ separate regression
problems. We also note that, the coherence assumption in
\cite{colt2009} was formulated with the numerical constant $7$
instead of $14$. The larger constant here is due to the fact that we
consider the general model with not necessarily block diagonal
design matrix, in contrast to the multi-task setting of
\cite{colt2009}.

Lemma \ref{lem:2}, which is presented in the appendix, establishes
that Assumption \ref{MC} implies Assumption \ref{RE}. Note also
that, by an argument as in \cite{lounici2008snc}, it is not hard to
show that under Assumption \ref{MC} any group $s$-sparse vector
$\beta^*$ satisfying (\ref{eq:model}) is unique.

Theorem \ref{th1} provides bounds for compound measures of risk,
that is, depending simultaneously on all the vectors $\beta^j$. An
important question is to evaluate the performance of estimators for
each of the components $\beta^j$ separately. The next theorem
provides a bound of this type and, as a consequence, a result on the
selection of sparsity pattern.

\begin{theorem}\label{est-supnorm}
Let the assumptions of Theorem \ref{th1} be satisfied and let
Assumption \ref{MC} hold with the same $s$. Set
\begin{equation}\label{eq:c}
c = \left(\frac{3}{2} + \frac{16}{7(\alpha-1)} \right).
\end{equation}
Then with probability at least $1 - 2M^{1-q}$, for any solution
$\hat{\beta}$ of problem (\ref{eq:opt}) we have that
\begin{equation}\label{CI}
\|\hat{\beta}-\beta^{*}\|_{2,\infty} \leq \frac{c}{\phi}
\lambda_{\max}.
\end{equation}
If, in addition,
\begin{equation}\label{lb}
\min_{j\in J(\beta^{*})} \|(\beta^{*})^j\| >
\frac{2c}{\phi}\lambda_{\max},
\end{equation}
then with the same probability for any solution $\hat{\beta}$ of
problem (\ref{eq:opt}) the set of indices
\begin{equation}\label{lb1}
\hat{J} = \left\{j: \|\hat{\beta}^j\| >
\frac{c}{\phi}\lambda_{\max}\right\}
\end{equation}
estimates correctly the sparsity pattern $J(\beta^{*})$, that is,
$$
\hat{J} = J(\beta^{*}).
$$
\end{theorem}

\begin{proof}
Set $K_\infty = \max_{1 \leq j \leq M} K_j$. We define first for any
$j,j'\in \mathbb N_M$ the $K_\infty \times K_\infty$ matrix
$\tilde{\Psi}[j,j']$ as follows. If $j\neq j'$ we have
$(\tilde{\Psi}[j,j'])_{t \in \N_{K_j}, t'\in \N_{K_{j'}}}
= \Psi[j,j']$ and $(\tilde{\Psi}[j,j'])_{t, t'}=0$ if $t
>K_j$ or if $t'>K_{j'}$. If $j=j'$ we have $(\tilde{\Psi}[j,j])_{t,t'\in \N_{K_j}}
= \Psi[j,j] - \phi I_{K_j \times K_j}$ and
$(\tilde{\Psi}[j,j])_{t, t'}=0$ if $t
>K_j$ or if $t'>K_{j}$. Similarly, for any $\Delta \in \mathbb
R^K$ and any $j\in \mathbb N_M$ we set $\tilde{\Delta}^j\in \mathbb
R^{K_\infty}$ such that $(\tilde{\Delta}^j_{t})_{t \in \N_{K_j}} = \Delta^j$ and $\tilde{\Delta}^j_{t} =0$ for any $t >K_j$.

Set $\Delta = \hat{\beta}-\beta^{*}$. We have
\begin{align}\label{interm-supnorm-1}
  \phi\|\Delta\|_{2,\infty} &\leq \|\Psi\Delta\|_{2,\infty} +
\|(\Psi-\phi I_{K\times K})\Delta\|_{2,\infty}.
\end{align}

Using Cauchy-Schwarz's inequality we obtain
\begin{align}\label{interm-supnorm-2}
\|(\Psi-\phi I_{K\times K})\Delta \|_{2,\infty} &=\max_{1\leq j \leq
M} \left[\sum_{t =
1}^{K_j}\left(\sum_{j'=1}^{M}\sum_{t'=1}^{K_{j'}}\left(\tilde{\Psi}[j,j']\right)_{t,t'}
\tilde{\Delta}_{t'}^{j'} \right)^{2}\right]^{1/2}
 \nonumber\\
&\leq \max_{1\leq j \leq M} \left[\sum_{t =
1}^{K_j}\left(\sum_{j'=1}^{M}\left(\tilde{\Psi}[j,j']\right)_{t,t}
\tilde{\Delta}_{t}^{j'} \right)^{2}\right]^{1/2}\nonumber\\
&\hspace{.0cm}+ \max_{1\leq j \leq M}
\left[\sum_{t=1}^{K_j}\left(\sum_{j'=1}^{M}\sum_{t'=1,t'\neq
t}^{K_{j'}}\left(\tilde{\Psi}[j,j']\right)_{t,t'}
\tilde{\Delta}_{t'}^{j'} \right)^{2}\right]^{1/2}\hspace{-.5truecm}.
\end{align}
We now treat the first term on the right-hand side of
(\ref{interm-supnorm-2}). We have, using Assumption \ref{MC} and
Minkowski's inequality for the Euclidean norm in $\mathbb R^{K_j}$,
that
\begin{align*}
\max_{1\leq j \leq M} \left[\sum_{t=
1}^{K_j}\left(\sum_{j'=1}^{M}\left(\tilde{\Psi}[j,j']\right)_{t,t}
\tilde{\Delta}_{t}^{j'} \right)^{2}\right]^{1/2} &\leq
\frac{\lambda_{\min }\phi}{14\alpha\lambda_{\max}s} \left[\sum_{t
= 1}^{K_j}\left(\sum_{j'=1}^{M}|\tilde{\Delta}_{t}^{j'}|
\right)^{2}\right]^{1/2}\\
&\leq \frac{\lambda_{\min }\phi}{14\alpha\lambda_{\max}s}
\|\tilde{\Delta}\|_{2,1} \\
& \leq \frac{\lambda_{\min
}\phi}{14\alpha\lambda_{\max}s} \|\Delta\|_{2,1},\nonumber
\end{align*}
since $\|\tilde{\Delta}\|_{2,1} \leq \|\Delta\|_{2,1}$ by definition
of $\tilde{\Delta}$. Next we treat the second term in the right-hand
side of (\ref{interm-supnorm-2}). Cauchy-Schwarz's inequality gives
\begin{align*}
&\max_{1\leq j \leq M}
\left[\sum_{t=1}^{K_j}\left(\sum_{j'=1}^{M}\sum_{t'=1,t'\neq
t}^{K_{j'}}\left(\tilde{\Psi}[j,j']\right)_{t,t'}
\tilde{\Delta}_{t'}^{j'} \right)^{2}\right]^{1/2}\nonumber\\
&\hspace{3cm}\leq \frac{\lambda_{\min}\phi}{14\alpha
\lambda_{\max}s}\max_{1\leq j \leq M} \left[
\frac{1}{K_{j}}\sum_{t=1}^{K_j}
\left(\sum_{j'=1}^{M}\sum_{t'=1}^{K_{j'}}\frac{|\tilde{\Delta}_{t'}^{j'}|}{\sqrt{K_{j'}}}
\right)^{2} \right]^{1/2}
 \nonumber\\
&\hspace{3cm}\leq
 \frac{\lambda_{\min}\phi}{14\alpha
\lambda_{\max}s}\sum_{j'=1}^{M}\sum_{t'=1}^{K_{j'}}\frac{|\tilde{\Delta}_{t'}^{j'}|}{\sqrt{K_{j'}}}
\nonumber\\
&\hspace{3cm}\leq \frac{\lambda_{\min}\phi}{14\alpha \lambda_{\max}
s}\|\tilde{\Delta}\|_{2,1} \leq \frac{\lambda_{\min}\phi}{14\alpha
\lambda_{\max} s}\|\Delta\|_{2,1}.
\end{align*}

Combining the four above displays we get
\begin{equation*}
\|\Delta\|_{2,\infty} \leq \frac{1}{\phi}\|\Psi\Delta\|_{2,\infty} +
\frac{2\lambda_{\min}}{14\alpha \lambda_{\max} s }\|\Delta\|_{2,1}.
\end{equation*}
Thus, by inequalities \eqref{eq:2} and \eqref{eq:t2-GL},
with probability at least $1 - 2M^{1-q}$, it holds that
\begin{equation*}
\|\Delta\|_{2,\infty} \leq \left(\frac{3}{2\phi} +\frac{16}{7 \alpha
\kappa^{2}}\right)\lambda_{\max}.
\end{equation*}
By Lemma \ref{lem:2}, $\alpha \kappa^{2}=(\alpha-1)\phi$, which
yields the first result of the theorem. The second result follows
from the first one in an obvious way.
\end{proof}

Assumption of type (\ref{lb}) is inevitable in the context of
selection of sparsity pattern. It says that the vectors
$(\beta^*)^j$ cannot be arbitrarily close to 0 for $j$ in the
pattern. Their norms should be at least somewhat larger than the
noise level.

Theorems \ref{th1} and \ref{est-supnorm} imply the following
corollary.

\begin{corollary}\label{p-norm}
Let the assumptions of Theorem \ref{th1} be satisfied and let
Assumption \ref{MC} hold with the same $s$. Then with probability at
least $1 - 2M^{1-q}$, for any solution $\hat{\beta}$ of problem
(\ref{eq:opt}) and any $1\le p <\infty$ we have that
\begin{equation}\label{eq:pnorm}
\|\hat{\beta}-\beta^{*}\|_{2,p} \leq
\frac{c_1}{\phi}\lambda_{\max}\left( \sum_{j\in
J(\beta^*)}\frac{\lambda_j^2}{\lambda_{\min}\lambda_{\max}}\right)^{\frac{1}{p}}
,
\end{equation}
where
\begin{equation}\label{eq:c1} c_{1} = \left(
\frac{16\alpha}{\alpha-1} \right)^{1/p}\left(\frac{3}{2}+
\frac{16}{7(\alpha-1)} \right)^{1-\frac{1}{p}} .
\end{equation}
 If, in addition, (\ref{lb}) holds, then
with the same probability for any solution $\hat{\beta}$ of problem
(\ref{eq:opt}) and any $1\le p <\infty$ we have that
\begin{equation}\label{CI1}
\|\hat{\beta}-\beta^{*}\|_{2,p} \leq
\frac{c_1}{\phi}\lambda_{\max}\left( \sum_{j\in \hat
J}\frac{\lambda_j^2}{\lambda_{\min}\lambda_{\max}}\right)^{\frac{1}{p}},
\end{equation}
where $\hat J$ is defined in (\ref{lb1}).
\end{corollary}

\begin{proof}
Set $\Delta = \hat{\beta} - \beta$. For any $p\geq 1$ we use the
norm interpolation inequality
\begin{equation*}
\| \Delta \|_{2,p} \leq
 \| \Delta \|_{2,1} ^{\frac{1}{p}}
\| \Delta \|_{2,\infty} ^{1-\frac{1}{p}}.
\end{equation*}
Combining inequalities \eqref{eq:t2-GL} and \eqref{CI} with $\kappa
= \sqrt{(1-1/\alpha)\phi}$ (cf. Lemma \ref{lem:2}) and the last
inequality yields (\ref{eq:pnorm}). Inequality (\ref{CI1}) is then
straightforward in view of Theorem \ref{est-supnorm}.
\end{proof}

Note that we introduce inequalities (\ref{CI}) and (\ref{CI1}) valid
with probability close to 1 because their right-hand sides are data
driven, and so they can be used as confidence bands for the unknown
parameter $\beta^*$ in mixed (2,$p$)-norms.



We finally derive a corollary of Theorem \ref{est-supnorm} for the
multi-task setting, which is straightforward in view of the above
results. 

\begin{corollary}\label{cor-MT-pattern}
Consider the multi-task model \eqref{eq:s1} for $M \geq 2$ and $T,n
\geq 1$. Let the assumptions of Theorem \ref{est-supnorm} be
satisfied and set
$$
\lambda = \frac{2\sqrt{2}\sigma}{\sqrt{nT}}\left(1 + \frac{A \log
M}{T}\right)^{1/2},
$$
where $A>5/2$. Then with probability at least $1-2M^{1-2A/5}$, for
any solution $\hat{\beta}$ of problem (\ref{eq:opt-MT}) and any
$1\le p\le \infty$ we have
\begin{equation}
\frac{1}{\sqrt{T}}\|\hat{\beta}-\beta^{*}\|_{2,p} \leq
\frac{2\sqrt{2}c_1\sigma s^{1/p}}{\sqrt{n}}\left(1 + \frac{A \log
M}{T}\right)^{1/2}\,,
\end{equation}
where $c_1$ is the constant defined in (\ref{eq:c1}) and we set
$x^{1/\infty}=1$ for any $x>0$. If, in addition,
\begin{equation}
\min_{j\in J(\beta^{*})}\frac{1}{\sqrt{T}} \|(\beta^{*})^j\| >
\frac{4\sqrt{2}c\sigma}{\sqrt{n}}\left(1 + \frac{A \log
M}{T}\right)^{1/2}\,,
\end{equation}
then with the same probability for any solution $\hat{\beta}$ of
problem (\ref{eq:opt-MT}) the set of indices
\begin{equation}
\hat{J} = \left\{j: \frac{1}{\sqrt{T}}\|\hat{\beta}^j\| >
\frac{2\sqrt{2}c\sigma}{\sqrt{n}}\left(1 + \frac{A \log
M}{T}\right)^{1/2}\,\right\}
\end{equation}
estimates correctly the sparsity pattern $J(\beta^{*})$, that is,
$$
\hat{J} = J(\beta^{*}).
$$
\end{corollary}


\section{Minimax lower bounds for arbitrary estimators}
\label{sec:minimax_lb}

In this section we consider again the multi-task model as in
Sections \ref{sec:MT-model} and \ref{Section2.1}. We will show that
the rate of convergence 
obtained in Corollary \ref{co1} is optimal in a minimax sense (up to
a logarithmic factor) for all estimators over a class of group
sparse vectors. 
This will be done under the following mild condition on matrix $X$. 
\begin{assumption}\label{SPD}
There exist positive constants $\kappa_1$ and $\kappa_2$ such that
for any vector $\Delta\in \mathbb R^{MT}\setminus \{0\}$ with
$M(\Delta)\leq 2s$ we have
$$
(a) \quad \frac{\|X\Delta\|^2}{n\|\Delta\|^2}\ge \kappa_1^2, \quad
\quad \quad \quad (b)\quad \frac{\|X\Delta\|^2}{n\|\Delta\|^2} \leq
\kappa_2^2.
$$
\end{assumption}
Note that part (b) of Assumption~\ref{SPD} is automatically
satisfied with $\kappa_2^2=\phi_{\rm MT}$ where $\phi_{\rm MT}$ is
the spectral norm of matrix $X^{\trans}X/n$. The reason for
introducing this assumption is that the $2s$-restricted maximal
eigenvalue $\kappa_2^2$ can be much smaller than the spectral norm
of $X^{\trans}X/n$, which would result in a sharper lower bound, see
Theorem \ref{th_minimax_lb} below.

In what follows we fix $T\ge1, M\ge2$, $s\le M/2$ and denote by
$GS(s,M,T)$ the set of vectors $\beta\in \R^{MT}$ such that
$M(\beta)\leq s$. Let $\ell\,:\, \mathbb R^{+} \rightarrow \mathbb
R^{+}$ be a nondecreasing function such that $\ell(0)= 0$ and $\ell
\not\equiv 0$.

\begin{theorem} \label{th_minimax_lb} Consider the multi-task model \eqref{eq:s1} for $M \geq 2$
and $T,n \geq 1$. Assume that $W\in\R^N$ is a random vector with
i.i.d. $\GD(0,\sigma^2)$ gaussian components, $\sigma^2>0$. Suppose
that $s\leq M/2$ and let part (b) of Assumption~\ref{SPD} be
satisfied. Define
$$
\psi_{n,p} = \frac{\sigma}{\kappa_2}\frac{s^{1/p}}{\sqrt{n}}
\left(1+ \frac{\log (eM/s)}{T} \right)^{1/2}, \quad \quad 1\le p\le
\infty,
$$
where we set $s^{1/\infty}=1$. Then there exist positive constants
${\bar b}, {\bar c}$ depending only on $\ell(\cdot)$ and $p$ such
that
\begin{equation}\label{eq:minm_lb2}
 \inf_{\tau}\sup_{\beta^*\in GS(s,M,T)}\mathbb
 E \ell\left({\bar b}\psi_{n,p}^{-1}\frac{1}{\sqrt{T}}
 \|\tau - \beta^* \|_{2,p}\right)
  \geq {\bar c},
 \end{equation}
where $\inf_{\tau}$ denotes the infimum over all estimators ${\tau}$
of $\beta^*$. If, in addition, part (a) of Assumption~\ref{SPD} is
satisfied, then there exist positive constants ${\bar b}, {\bar c}$
depending only on $\ell(\cdot)$ such that
  \begin{equation}\label{eq:minm_lb1}
 \inf_{\tau}\sup_{\beta^*\in GS(s,M,T)}\mathbb
 E \ell\left({\bar b}\psi_{n,2}^{-1}\frac{1}{\kappa_1\sqrt{nT}}\|X(\tau - \beta^*) \|\right)
  \geq {\bar c}.
 \end{equation}
\end{theorem}

\begin{proof} Fix $p$ and write for brevity $\psi_{n}=\psi_{n,p}$
where it causes no ambiguity. Throughout this proof we set
$x^{1/\infty}=1$ for any $x\ge0$. We consider first the case $T\leq
\log(eM/s)$. Set $\textbf{0} = (0,\ldots,0)\in\mathbb R^T$,
$\textbf{1} = (1,\ldots,1)\in\mathbb R^T$. Define the set of vectors
$$
\Omega = \left\lbrace \omega \in \mathbb
  R^{MT}\, : \, \omega^j \in \{\textbf{0},\textbf{1}\},\, j=1,\dots,M,\; \text{and}\; M(\omega)\leq s
  \right\rbrace,
$$
and its dilation
  $$
    \mathcal C (\Omega)= \left\lbrace \gamma\psi_{n,p}
    \omega /s^{1/p}\,:\,
    \omega\in \Omega
    \right\rbrace,
  $$
where $\gamma >0$ is an absolute constant to be chosen later. Note
that $\mathcal C(\Omega) \subset GS(s,M,T)$.

For any $\omega,\omega'$ in $\Omega$ we have $M(\omega-\omega') \le
2s$. Thus, for $\beta=\gamma\psi_{n,p} \omega /s^{1/p},\
\beta'=\gamma\psi_{n,p} \omega'/s^{1/p}$ parts (a) and (b) of
Assumption \ref{SPD} imply respectively
\begin{equation}\label{eq:minm_lba}
\frac{1}{n}\|X\beta - X\beta'\|^2 \ge
\frac{\kappa_1^2\gamma^2\psi_{n,p}^2
\rho(\omega,\omega')T}{s^{2/p}},
\end{equation}
\begin{equation}\label{eq:minm_lbb}
\frac{1}{n}\|X\beta - X\beta'\|^2 \leq
\frac{\kappa_2^2\gamma^2\psi_{n,p}^2 \rho(\omega,\omega')T}{s^{2/p}}
\end{equation}
where $\rho(\omega,\omega')= \sum_{j=1}^MI\{\omega^j\neq
(\omega')^j\}$ and $I\{\cdot\}$ denotes the indicator function. This
and the definition of $\psi_{n,p}$ yield that if part (a) of
Assumption~\ref{SPD} holds, then for all $\omega,\omega'\in\Omega$
we have
\begin{equation}\label{eq:minm_lb3}
\frac{1}{nT}\|X\beta - X\beta'\|^2 \geq \gamma^2\frac{\kappa_1^2
\sigma^2}{\kappa_2^2 n}\left(1+ \frac{\log(eM/s)}{T}\right)
\rho(\omega,\omega').
\end{equation}
Also, by definition of $\beta,\beta'$,
\begin{equation}\label{eq:minm_lb4} \frac{1}{\sqrt{T}}
\|\beta - \beta'\|_{2,p} = \frac{\gamma\sigma}{\kappa_2\sqrt{
n}}\left(1+ \frac{\log(eM/s)}{T}\right)^{1/2}
\left(\rho(\omega,\omega')\right)^{1/p}I\{\omega\neq (\omega')\}.
\end{equation}
For $\theta\in \R^N$, we denote by $P_\theta$ the probability
distribution of $\GD(\theta, \sigma^2 I_{N\times N})$ Gaussian
random vector. We denote by ${\cal K}(P,Q)$ the Kullback-Leibler
divergence between the probability measures $P$ and $Q$. Then, under
part (b) of Assumption~\ref{SPD},
\begin{eqnarray}\label{eq:minm_lb5}
  {\cal K}(P_{X\beta},P_{X\beta'}) 
  &=& \frac{1}{2\sigma^2}\|X\beta - X\beta'\|^2\nonumber\\
  &\leq& \frac{\kappa_2^2 \gamma^2}{2\sigma^2 s^{2/p} } \,n
  \psi_{n,p}^2 \rho(\omega,\omega')T\nonumber\\
  &\leq& \gamma^2 s [T+ \log(eM/s)]\nonumber\\
  &\leq& 2\gamma^2 s \log(eM/s)
\end{eqnarray}
where we used that $\rho(\omega,\omega')\le 2s$ for all $\omega,
\omega'\in\Omega$. Lemma 8.3 in \cite{RigTsy10} guarantees the
existence of a subset $\mathcal N$ of $\Omega$ such that
\begin{eqnarray}\label{eq:minm_lb6}
    \log(|\mathcal N|)&\geq& \tilde{c}s\log\left(\frac{eM}{s}
    \right)\\
    \rho(\omega,\omega')&\geq& s/4,\forall \omega,\omega' \in \mathcal{N},\ \omega\neq
\omega',\nonumber
\end{eqnarray}
for some absolute constant $\tilde{c}>0$, where $|\mathcal N|$
denotes the cardinality of $\mathcal N$. Combining this with
(\ref{eq:minm_lb3}) and (\ref{eq:minm_lb4}) we find that the finite
set of vectors $\mathcal C(\mathcal N)$ is such that, for all
$\beta,\beta'\in \mathcal C(\mathcal N)$, $\beta\ne\beta'$,
$$ \frac{1}{\sqrt{T}}\|\beta - \beta'\|_{2,p} \geq \frac{\gamma\sigma s^{1/p}}{4^{1/p}\kappa_2 \sqrt{n}}\left(1+
\frac{\log(eM/s)}{T}\right)^{1/2}=\frac{\gamma}{4^{1/p}}\psi_{n,p}\,,
$$
and under part (a) of Assumption~\ref{SPD},
$$
\frac{1}{nT}\|X\beta - X\beta'\|^2 \geq \gamma^2\frac{\kappa_1^2
\sigma^2 s}{4\kappa_2^2 n}\left(1+ \frac{\log(eM/s)}{T}\right)=
\frac{\gamma^2}{4}\kappa_1^2\psi_{n,2}^2\,.
$$
Furthermore, by (\ref{eq:minm_lb5}) and (\ref{eq:minm_lb6}) for all
$\beta,\beta'\in \mathcal C(\mathcal N)$ under part (b) of
Assumption~\ref{SPD} we have
$$
{\cal K}(P_{X\beta},P_{X\beta'}) \leq \frac{1}{16}
\log\left(|\mathcal N|\right) = \frac{1}{16} \log\left(|\mathcal
C(\mathcal N)|\right)
$$
for an absolute constant $\gamma>0$ chosen small enough. Thus, the
result follows by application of Theorem 2.7 in \cite{tsy09}.

Consider now the case $T > \log(eM/s)$. Introduce the set of vectors
$$
 \Omega' = \left\lbrace \omega\in\mathbb R^{MT}\, : \,
 \omega = (\omega^1,\ldots,\omega^{M}),~
\omega^j \in \{0,1\}^T~\text{if}~ j \leq s~\text{and}~\omega^j =
{\bf 0}~\text{otherwise}\right\rbrace,
$$
and the associated dilated set $\mathcal C(\Omega')$ defined as
above. Note that $\mathcal C(\Omega') \subset GS(s,M,T)$.

For any $\omega,\omega' \in \Omega'$ we define
$\rho'(\omega,\omega') = \sum_{j=1}^{M}\sum_{t=1}^{T}I\{\omega_{tj}
\neq \omega'_{tj}\}=\sum_{j=1}^{s}\sum_{t=1}^{T}I\{\omega_{tj} \neq
\omega'_{tj}\}$.

We assume first that $Ts\ge 8$. Then Varshamov-Gilbert Lemma (see
Lemma 2.9 in \cite{tsy09}) guarantees that there exists a subset
$\mathcal N'$ of $\Omega'$ such that
\begin{eqnarray}
  \label{eq:minm_lb7}
  |\mathcal N'| &\geq& 2^{Ts/8},\\
  \label{eq:minm_lb7b}
    \rho'(\omega,\omega') &\geq& \frac{Ts}{8}, \forall \omega,\omega' \in \mathcal{N'},\ \omega\neq
\omega'.
  \nonumber
\end{eqnarray}
Next for any $\omega, \omega' \in \mathcal N'$ we have
$M(\omega-\omega')\le 2s$, and thus under parts (a) and (b) of
Assumption~\ref{SPD} we have, respectively,
$$
\frac{1}{n}\|X\beta - X\beta'\|^2\ge
\frac{\kappa_1^2\gamma^2\psi_{n}^2
\rho'(\omega,\omega')}{s^{2/p}}\,, \quad \quad \frac{1}{n}\|X\beta -
X\beta'\|^2 \leq \frac{\kappa_2^2 \gamma^2\psi_{n}^2
\rho'(\omega,\omega')}{ s^{2/p}}
$$
where $\beta=\gamma\psi_{n} \omega /s^{1/p},\ \beta'=\gamma\psi_{n}
\omega'/s^{1/p}$ are any two elements of $\mathcal C(\mathcal N')$.

Now, using Lemma \ref{ineq_p} in the Appendix we get that, for all
$\omega,\omega' \in \mathcal{N'}$ such that $\omega\neq \omega'$,
\begin{equation}\label{eq:Varsham}
\|\omega - \omega'\|_{2,p} \ge
\left(\frac{s}{16}\right)^{1/p}\frac{\sqrt{T}}{4}\,, \quad \forall \
\ 1\le p \le \infty.
\end{equation}
Thus, for all $\beta,\beta'\in \mathcal C(\mathcal N')$ such that
$\beta\ne\beta'$ we have
\begin{eqnarray*}\frac{1}{\sqrt{T}}\|\beta -
\beta'\|_{2,p} &=& \frac{\gamma\psi_{n}}{s^{1/p}\sqrt{T}}\|\omega -
\omega'\|_{2,p} \ge \frac{\gamma}{16^{1/p}4}\psi_n
\end{eqnarray*}
(recall that $\psi_n=\psi_{n,p}$), and under part (a) of
Assumption~\ref{SPD},
\begin{eqnarray*}
    \frac{1}{nT}\|X\beta -
X\beta'\|^2 &\geq&  \frac{\gamma^2}{8}\frac{s\kappa_1^2
\sigma^2}{\kappa_2^2 n}\left(1+ \frac{\log(eM/s)}{T}\right)=
\frac{\gamma^2}{8}\kappa_1^2\psi_{n,2}^2.
\end{eqnarray*}
Furthermore, for all $\beta,\beta'\in \mathcal C(\mathcal N')$ under
part (b) of Assumption~\ref{SPD},
\begin{eqnarray*}
{\cal K}(P_{X\beta},P_{X\beta'}) &\leq& 2\gamma^2 sT \leq
\frac{1}{16}\log(|C(\mathcal N')|),
\end{eqnarray*}
where, in view of (\ref{eq:minm_lb7}), the last inequality holds for
an absolute constant $\gamma>0$ chosen small enough. We apply again
Theorem 2.7 in \cite{tsy09} to get the result.

Finally, if $T> \log(eM/s)$ and $Ts< 8$, then the rate $\psi_n$ is
of the order $1/n$. This is the standard parametric rate and the
lower bounds are easily obtained by reduction to distinguishing
between two elements of $GS(s,M,T)$.
\end{proof}

As a consequence of Theorem~\ref{th_minimax_lb}, we get, for
example, the lower bounds for the squared loss $\ell(u)=u^2$ and for
the indicator loss $\ell(u)=I\{u\ge 1\}$. The indicator loss is
relevant for comparison with the upper bounds of
Corollaries~\ref{co1} and~\ref{cor-MT-pattern}. For example,
Theorem~\ref{th_minimax_lb} with this loss and $p=1,2$ implies that
there exists $\beta^*\in GS(s,M,T)$ such that, for any estimator
$\tau$ of $\beta^*$,
$$
\frac{1}{\sqrt{nT}}\|X(\tau - \beta^*) \| \ge C\sqrt{\frac{s}{n}}
\left(1+ \frac{\log (eM/s)}{T} \right)^{1/2}
$$
and
$$
\frac{1}{\sqrt{T}}\|\tau - \beta^*\| \ge C\sqrt{\frac{s}{n}}
\left(1+ \frac{\log (eM/s)}{T} \right)^{1/2}, \quad
\frac{1}{\sqrt{T}}\|\tau - \beta^*\|_{2,1} \ge C\frac{s}{\sqrt{n}}
\left(1+ \frac{\log (eM/s)}{T} \right)^{1/2}
$$
with a positive probability (independent of $n, s,M,T$) where $C>0$
is some constant. The rate on the right-hand side of these
inequalities is of the same order as in the corresponding upper
bounds in Corollary~\ref{co1}, modulo that $\log M$ is replaced here
by $\log (eM/s)$. We conjecture that the factor $\log (eM/s)$ and
not $\log M$ corresponds to the optimal rate; actually, we know that
this conjecture is true when $T=1$ and the risk is defined by the
prediction error with $\ell(u)=u^2$ \cite{RigTsy10}.

A weaker version of Theorem~\ref{th_minimax_lb}, with $\ell(u)=u^2$,
$p=2$ and suboptimal rate of the order $[s\log (M/s) /(nT)]^{1/2}$
is established in \cite{zhang}.

\begin{remark} For the model with usual (non-grouped) sparsity,
which corresponds to $T=1$, the set $GS(s,M,1)$ coincides with the
$\ell_0$-ball of radius $s$ in $\R^M$. Therefore,
Theorem~\ref{th_minimax_lb} generalizes the minimax lower bounds on
$\ell_0$-balls recently obtained in \cite{raskut09} and \cite{RigTsy10}
for the usual sparsity model. Those papers considered only the
prediction error and the $\ell_2$ error under the squared loss
$\ell(u)=u^2$. Theorem~\ref{th_minimax_lb} covers any $\ell_p$ error
with $1\le p\le \infty$ and applies with general loss functions
$\ell(\cdot)$. As a particular instance, for the indicator loss
$\ell(u)=I\{u\ge 1\}$ and $T=1$, the lower bounds of
Theorem~\ref{th_minimax_lb} show that the upper bounds for the
prediction error and the $\ell_p$ errors ($1\le p\le\infty$) of the
usual Lasso estimator established in \cite{BRT} and
\cite{lounici2008snc} cannot be improved in a minimax sense on
$\ell_0$-balls up to logarithmic factors. Note that this conclusion
cannot be deduced from the lower bounds of \cite{raskut09} and
\cite{RigTsy10}. 
\end{remark}


\section{Lower bounds for the Lasso}
\label{sec:lb}

In this section we establish lower bounds on the prediction and
estimation accuracy of the Lasso estimator. As a consequence, we can
emphasize the advantages of using the Group Lasso estimator as
compared to the usual Lasso in some important particular cases.

The Lasso estimator is a solution of the minimization problem \beq
\min_{\beta\in \mathbb R^K}\frac{1}{N}\|X\beta - y\|^2 +
2r\|\beta\|_{1}, \label{eq:lasso} \eeq where $\|\beta\|_1 =
\sum_{j=1}^K |\beta_j|$ and $r$ is a positive parameter. The
following notations apply only to this section. For any vector
$\beta\in \mathbb R^K$ and any subset $J \subseteq \N_K$, we denote
by $\beta_{|J}$ the vector in $\mathbb R^{K}$ which has the same
coordinates as $\beta$ on $J$ and zero coordinates on the complement
$J^{c}$ of $J$, $J'(\beta) = \{j:\beta_j\neq 0\}$ and $M'(\beta) =
|J'(\beta)|$.


We will use the following standard assumption on the matrix $X$ (the
Restricted Eigenvalue condition in \cite{BRT}).
\begin{assumption}
\label{RE-Lasso}
Fix $s'\geq 1$. There exists a positive number $\kappa'$ such that
\begin{align*}
\min \bigg\{ \frac{\|X\Delta\|}{\sqrt{N}\|\Delta_{|J}\|}~:~ |J| \leq
s', \Delta\in \R^{K}\setminus \{0\}, \,\sum_{j \in J^c} |\Delta_j|
\leq 3 \sum_{j \in J} |\Delta_j| \bigg\} \geq \kappa',
\end{align*}\label{ass-Lasso} where $J^c$ denotes the complement of
the set of indices~$J$.
\end{assumption}

\begin{theorem}\label{thm:lb}
 Let
Assumption \ref{RE-Lasso} be satisfied. Assume that $W\in\R^N$ is a
random vector with i.i.d. $\GD(0,\sigma^2)$ gaussian components,
$\sigma^2>0$. Set $r = A\sigma\sqrt{\frac{\phi\log K}{N}}$ where
$A>2\sqrt{2}$ and $\phi$ is the maximal diagonal element of the
matrix $\Psi=\frac{1}{N} X\trans X$. If $\hbeta^L$ is a solution of
problem \eqref{eq:lasso}, then with probability at least
$1-K^{1-\frac{A^2}{8}}$ we have
\begin{eqnarray}\label{eq:lb-1}
  \frac{1}{N}\|X(\hat{\beta}^L - \beta^*)\|^2 &\geq&
  M'(\hat{\beta}^L)\frac{A^2\sigma^2 \phi\log K}{4\phi_{\max}N}, \\
  \|\hat{\beta}^L - \beta^*\| &\geq&
  \frac{A\sigma}{2\phi_{\max}}\sqrt{M'(\hbeta^L)\frac{\phi\log K}{N}},
  \label{eq:lb-2}
  \end{eqnarray}
where $\phi_{\rm max}$ is the maximum eigenvalue of the matrix
$\Psi$.
%
If, in addition, $M'(\beta^*)\leq s'$, and \beq \min\{
|\Psi_{jj}\beta^*_j|: j \in \N_m,~\beta_j^*\neq 0\}> \left( \frac{3}{2} +
\frac{16s'}{\kappa'^2}\max_{j\neq k}|\Psi_{jk}|\right)r,
  \label{eq:lb-c}
\eeq
where $\Psi_{jk}$ denotes the $(j,k)$-th entry of matrix $\Psi$, then
with the same probability we have  \beq M'(\hbeta^L) \geq
M'(\beta^*). \label{eq:lb-3} \eeq
\end{theorem}
\begin{proof} Inequality (B.3) in \cite{BRT} yields
 (\ref{eq:lb-1}) on the event $\mathcal{A} =
\left\{\frac{1}{N}\|X^{\trans}W\|_{\infty}\leq \frac{r}{2} \right\}$
of probability $\Prob (\mathcal{A})\ge 1-K^{1-\frac{A^2}{8}}$.

Next, \eqref{eq:lb-2} follows from \eqref{eq:lb-1} and the
inequality
$$
\frac{1}{N} (\hbeta^L -\beta^*)\trans X\trans X (\hbeta^L -\beta^*)
\leq \phi_{\rm max} \|\hbeta^L -\beta^*\|^2.
$$

We now prove (\ref{eq:lb-3}). If $M'(\hat{\beta}^L)< M'(\beta^*)$
then there exists $j\in J'(\hat{\beta}^L)^{c}\cap J'(\beta^*)$. Set
$\Delta = \beta^* - \hat{\beta}^L$ and recall that $\Psi =
\frac{1}{N}X^{\trans}X$. Using 
that any Lasso solution $\hat{\beta}^L$ satisfies
\begin{eqnarray}\label{KKT-lasso}
  \begin{cases}
    \frac{1}{N}(X^{\trans}(y - X\hat{\beta}^L ))_j = \mathrm{sign}(\hat{\beta}^L_j)r,&\text{if $\hat{\beta}^L_j\neq 0$,}\\
    \left|\frac{1}{N}(X^{\trans}(y - X\hat{\beta}^L ))_j\right|\leq r,&\text{if $\hat{\beta}^L_j =0$.}
  \end{cases}
\end{eqnarray}
and the triangle inequality we get, on the event $\mathcal{A}$, that
$|(\Psi\Delta)_j|\leq \frac{3r}{2}$. Consequently,
\begin{eqnarray}\label{interm-lb1}
  |\Psi_{jj}\beta^*_j| = |\Psi_{jj}\Delta_j| =
  \left|(\Psi\Delta)_j -
  \sum_{k\neq j}\Psi_{jk}\Delta_k \right|\leq
  \frac{3r}{2} + \|\Delta\|_1\,\max_{j\neq k}|\Psi_{jk}|.
\end{eqnarray}
Next, Corollary B.2 in \cite{BRT} yields that, on the event
$\mathcal{A}$,
\begin{eqnarray*}
  \|\Delta_{|J'(\beta^*)^c}\|_1 &\leq& 3 \|\Delta_{|J'(\beta^*)}\|_1.
\end{eqnarray*}
Thus, the Cauchy-Schwarz inequality, Assumption \ref{ass-Lasso} and \cite[Inequality (7.8)]{BRT} give that, on the event $\mathcal{A}$,
\begin{eqnarray}\label{interm-lb2}
  \|\Delta\|_1 &\leq& 4 \|\Delta_{|J'(\beta^*)}\|_1
  \leq 4\sqrt{s'}\|\Delta_{|J'(\beta^*)}\|  \leq \frac{4\sqrt{s'}}{\kappa'}(\Delta^{\trans}\Psi\Delta)^{1/2}\leq \frac{16 s'}{\kappa'^{2}}r.
\end{eqnarray}
Combining (\ref{interm-lb1}) and (\ref{interm-lb2}) yields, on the
event $\mathcal{A}$, that
$$
|\Psi_{jj}\beta^*_j| \leq \left( \frac{3}{2} +
\frac{16s'}{\kappa'^2}\max_{j\neq k}|\Psi_{jk}|\right)r,
$$
which contradicts the condition \eqref{eq:lb-c}.

\end{proof}
Let us emphasize that the Theorem \ref{thm:lb} establishes lower
bounds, which hold for {\rm every} Lasso solution if $\hat\beta_L$
is not unique.

Theorem \ref{thm:lb} highlights several limitations of the usual
Lasso as compared to the Group Lasso. Let us explain this point in
the multi-task learning case. There, the usual Lasso estimator
$\hat{\beta}^{L}$ is a solution of the following optimization
problem
$$
\min \left\{ \frac{1}{T} \sum_{t=1}^T \frac{1}{n} \|X_t \beta_t -
y_t\|^2 +
 2 r \sum_{t=1}^{T}\sum_{j=1}^{M}|\beta_{tj}|\right\}.
$$
By comparing the prediction error lower bound in Theorem
\ref{thm:lb} for this estimator with the corresponding upper bound
for Group Lasso estimator derived in Corollary \ref{co1}, we reach
the following conclusions.
\begin{itemize}

\item \emph{The usual Lasso does not enjoy any dimension
independence phenomenon as compared to the Group Lasso}.

In the multi-task learning setting we have $N=nT$, $K=MT$. Assume
that the tasks' design matrices are orthogonal, namely $X_t\trans
X_t/n = I_{M \times M}$ for every $t \in \N_T$. Hence, $\Psi = I_{TM
\times TM}/T$, so that $\phi_{\rm max} = \phi=1/T$ and
$\Psi_{jj}=1/T$ for all $j$. Let a special instance of group sparsity assumption be realized, namely, all vectors $\beta_t^*$ have
exactly $s$ non-zero entries at the same positions. Then,
$M(\beta^*)=s$ and $M'(\beta^*)=sT$. Moreover, condition
\eqref{eq:lb-c} simplifies to the requirement that
$$
\min_{j : \beta^*_j \ne 0} |\beta^*_j|
 \geq
\frac{3A\sigma}{2}\sqrt{\frac{\log(MT)}{n}}\,.$$ We conclude by
inequalities \eqref{eq:lb-1} and \eqref{eq:lb-3} that, with
probability at least $1-(MT)^{1-\frac{A^2}{8}}$,
\begin{equation}\label{lower_L}
\frac{1}{nT}\|X(\hat{\beta}^L - \beta^*)\|^2 \geq A^2\sigma^2
s\frac{\log (MT)}{4n}.
\end{equation}
This bound holds no matter what the number of tasks $T$ is. In
contrast, the bounds in Corollary \ref{co1} can be made independent
of the dimension $M$ and of the number of tasks $T$ as soon as
$T\geq \log M$. Specifically, under the above assumptions we have,
recalling Definition \ref{RE-MT}, that $\kappa_{\rm MT} \geq 1$ and
by \eqref{eq:t1}, with probability close to 1, every Group Lasso
solution $\hbeta$ satisfies
\begin{equation}\label{upper_GL}
\frac{1}{nT}\|X({\hbeta} - \beta^*)\|^2 \leq 128 \sigma^2
\frac{s}{n}\left(1 + \frac{A\log M}{T}\right).
\end{equation}

\item \emph{The Group Lasso achieves faster rates of convergence in some cases as compared to the usual Lasso}.
We consider separately two cases. The first one is already discussed
the preceding remark. It corresponds to $T\geq \log M$. Then the
upper bound for the Group Lasso (\ref{upper_GL}) is smaller than the
lower bound (\ref{lower_L}) for the Lasso by a logarithmic factor.
This factor can be large if $T$ is large, for example exponential in
$n$, so that (\ref{lower_L}) gives no convergence result for the
Lasso. The second case is $T< \log M$. Then the lower bound
(\ref{lower_L}) is of the order $s(\log M)/n$, while the upper bound
(\ref{upper_GL}) is of the order $s(\log M)/(nT)$. The ratio is of
the order $T$ in favor of the Group Lasso.
\end{itemize}
In (\ref{lower_L}) and (\ref{upper_GL}) we have only compared the
prediction errors of the two estimators. In view of inequality
(\ref{eq:t4}) and Theorem~\ref{thm:lb}, similar observations are valid
for the $\ell_2$ estimation errors.

\section{Non-Gaussian noise}
\label{sec:5} In this section, we show that the above results extend
to non-gaussian noise. We consider here the multi-task setting
described in Section \ref{sec:MT-model} and we only assume that the
components of random vector $W$ are independent with zero mean and
finite fourth moment $\E[W_{tj}^{4}]$. As we shall see the results
remain similar to those of the previous sections, though the
concentration effect is weaker.

We need the following technical assumption.
\begin{assumption}\label{tech-ass}
The matrix $X$ is such that
\begin{equation*}
\max_{t\in \mathbb{N}_{T} }\left(\frac{1}{n}\sum_{i=1}^{n}\max_{j\in
\mathbb{N}_{M}}|(x_{ti})_{j}|^{2}\right)\leq x_*^2\end{equation*}
for a finite constant $x_*$.
\end{assumption}

This assumption is quite mild. It is satisfied for example, if all
$(x_{ti})_{j}$ are bounded in absolute value by a constant uniformly
in $i,t,j$. We have the two following theorems.
\begin{theorem}\label{pred-est-RE}
Consider the model \eqref{eq:model} for any $M\ge2$, $T,n\geq 1$.
Assume that the components of random vector $W$ are independent with
zero mean, $\max_{t\in \mathbb{N}_{T}, j\in
\mathbb{N}_{M}}\E[W_{tj}^{4}]\leq b^{4}$, all diagonal elements of
the matrix $X\trans X/n$ are equal to $1$ and $M(\beta^*) \leq s$.
Let also Assumption \ref{tech-ass} be satisfied. Set
$$
\lambda = \frac{x_*b}{\sqrt{nT}} \left(1+
\frac{\lom}{\sqrt{T}}\right)^{1/2},
$$
with $\delta>0$. Then with probability at least $1 -
\frac{4\sqrt{\log (2M)}
  [(8\log(12M))^2+1]^{1/2}}{\lom}$, for any
solution $\hat{\beta}$ of problem \eqref{eq:opt-MT} we have
\begin{eqnarray}
\label{eq:t11-0}
\frac{1}{nT} \|X(\hat{\beta} -
\beta^*)\|^2 & \leq & \frac{4x_*b}{\sqrt{nT}} \left(1+
\frac{\lom}{\sqrt{T}}\right)^{1/2}\|\beta^*\|_{2,1}.
\end{eqnarray}
If, in addition, Assumption \ref{ass_MT} holds, then with the same
probability for any solution $\hat{\beta}$ of problem
\eqref{eq:opt-MT} we have
\begin{eqnarray}
\label{eq:t11} \hspace{-.2truecm}\frac{1}{nT} \|X(\hat{\beta} -
\beta^*)\|^2 \hspace{-.2truecm} & \leq & \hspace{-.2truecm}
\frac{16 x_*^{2}b^2}{\kappa_{\rm MT}^2}   \frac{s}{n} \left(1+
\frac{\lom}{\sqrt{T}}\right),\hspace{.3truecm} \\
\label{eq:t21} \frac1{\sqrt {T}}\|\hat{\beta} - \beta^*\|_{2,1}
\hspace{-.2truecm} & \leq & \hspace{-.2truecm}
\frac{16x_*b}{\kappa_{\rm MT}^2} \frac{s}{\sqrt{n}} \left(1+
\frac{\lom}{\sqrt{T}}\right)^{1/2}, \\
\label{eq:t31} \hspace{-.2truecm} M(\hat{\beta}) \hspace{-.2truecm}
& \leq & \hspace{-.2truecm}  \frac{64 \phi_{\rm MT}}{\kappa_{\rm MT}^2} s,
\end{eqnarray}
where $\phi_{\rm MT}$ is the largest eigenvalue of the matrix
$X\trans X/n$. If, in addition, $\kappa_{\rm MT}(2s)>0$, then with
the same probability for any solution $\hat{\beta}$ of problem
\eqref{eq:opt-MT} we have
\begin{eqnarray}
\nonumber
\hspace{-.2truecm}\frac{1}{\sqrt{T}} \|\hat{\beta} - \beta^*\|
\hspace{-.2truecm} & \leq & \hspace{-.2truecm} \frac{4\sqrt{10}\,x_*b
}{\kappa^2(2s)}  \sqrt{\frac{s}{n}} \left(1+
\frac{\lom}{\sqrt{T}}\right)^{1/2}\,. \hspace{.3truecm}
\end{eqnarray}
\label{thm:6.1}
\end{theorem}

\begin{theorem}
Consider the model (\ref{eq:model}) for $M\ge2$, $T,n\geq 1$. Let
the assumptions of Theorem \ref{pred-est-RE} be satisfied and let
Assumption \ref{MC} hold with the same $s$. Set
$$
\tilde{c} = \left(\frac{3}{2} + \frac{8}{7(\alpha-1)} \right)x_*b.
$$
Let $\lambda$ be as in Theorem \ref{pred-est-RE}. Then with
probability at least $1 - \frac{4\sqrt{\log (2M)}
  [(8\log(12M))^2+1]^{1/2}}{\lom}$, for any solution $\hat{\beta}$
  of problem (\ref{eq:opt-MT}) we have
\begin{equation*}
\frac{1}{\sqrt{T}}\|\hat{\beta}-\beta^{*}\|_{2,\infty} \leq
 \frac{\tilde{c}}{\sqrt{n}} \left(1+
\frac{\lom}{\sqrt{T}}\right)^{1/2}.
\end{equation*}
If, in addition, it holds that
$$\min_{j\in J(\beta^{*})}\frac{1}{\sqrt{T}}
\|(\beta^{*})^j\| >  \frac{2\tilde{c}}{\sqrt{n}} \left(1+
\frac{\lom}{\sqrt{T}}\right)^{1/2},
$$
then with the same probability for any solution $\hat{\beta}$ of
problem (\ref{eq:opt-MT}) the set of indices
$$
\hat{J} = \Big\{j: \frac{1}{\sqrt{T}}\|\hat{\beta}^j\| >
\frac{\tilde{c}}{\sqrt{n}} \left(1+ \frac{\lom}{\sqrt{T}}\right)^{1/2}\Big\}
$$ estimates correctly the sparsity
pattern $J(\beta^{*})$:
$$
\hat{J} = J(\beta^{*}).
$$
\label{thm:6.2}
\end{theorem}

\begin{proof}
The proofs of these theorems are similar to those of Theorems
\ref{th1} and \ref{est-supnorm} up to a modification of the bound on
$\Prob(\mathcal{A}^{c})$ in Lemma \ref{lem:1}. We consider now the
event
\begin{equation*}
\mathcal{A}=\left\{\max_{j=1}^M \sqrt{\sum_{t=1}^T
\left(\sum_{i=1}^n (x_{ti})_j W_{ti} \right)^2}\leq \lambda nT
\right\}.
\end{equation*}
Define the random variables
\begin{equation*}
Y_{tj} = \left(\sum_{i=1}^{n}(x_{ti})_{j}W_{ti}\right)^{2} -
\sum_{i=1}^{n}|(x_{ti})_{j}|^{2}\E[ W_{ti}^{2}],\; \ j=1,\dots,M, \
t=1,\dots,T.
\end{equation*}
We have
\begin{eqnarray*}
\Prob(\mathcal{A}^{c})&=& \Prob\left(\max_{1\leq j \leq
M}\sum_{t=1}^{T}\left(\sum_{i=1}^{n}(x_{ti})_{j}W_{ti}\right)^{2}
\geq
(\lambda nT)^{2}\right)\\
&\leq& \Prob\left(\max_{1\leq j \leq M}\sum_{t=1}^{T}Y_{tj}\geq
x_*^{2}b^{2}n\sqrt{T}
\lom\right)\\
&\leq& \frac{\E\ \max_{1\leq j \leq
M}\left|\sum_{t=1}^{T}Y_{tj}\right| }{x_*^{2}b^{2}n\sqrt{T} \lom}
\,.
\end{eqnarray*}
Applying the maximal moment inequality of Lemma \ref{nem-sara} below
with $m=1$ and constant $c(1)=2$ we obtain
\begin{eqnarray}\label{eq611}
  \E\ \max_{1\leq j \leq
M}\left|\sum_{t=1}^{T}Y_{tj}\right|  &\le& \sqrt{8\log (2M)}
\,\,\E\left(\left[\sum_{t=1}^{T}\max_{1\leq j \leq
M}Y_{tj}^{2}\right]^{1/2}\right)\\
&\le& \sqrt{8\log (2M)} \,\, \left[\sum_{t=1}^{T}\E\left(\max_{1\leq
j \leq M}Y_{tj}^2\right)\right]^{1/2}\nonumber
\\
&\le& 4\sqrt{\log (2M)} \,\, \left\{b^4x_*^4n^2T
+\sum_{t=1}^{T}\E\left(\max_{1\leq j \leq
M}\left|\sum_{i=1}^{n}(x_{ti})_{j}W_{ti}\right|^{4}\right)
\right\}^{1/2}\,.\nonumber
\end{eqnarray}
By the maximal moment inequality of Lemma \ref{nem-sara} with $m=4$
and constant $c(4)=12$ (since $M\ge2$) the last expectation is
bounded, for any $t=1,\dots,T$, as
\begin{eqnarray*}
  \E\left(\max_{1\leq j \leq
  M}\left|\sum_{i=1}^{n}(x_{ti})_{j}W_{ti}\right|^{4}\right)
&\leq& (8\log (12 M) )^{2} \E\left(\left[ \sum_{i=1}^{n}\max_{1\leq
j \leq M}(x_{ti})_{j}^{2}W_{ti}^{2} \right]^{2}\right).
\end{eqnarray*}
Setting for brevity ${\bar x}_i= \max_{1\leq
j \leq M}(x_{ti})_{j}^{2}$ we have
\begin{eqnarray*}
\E\left(\left[ \sum_{i=1}^{n}\max_{1\leq j \leq
M}(x_{ti})_{j}^{2}W_{ti}^{2} \right]^{2}\right)&\le&
b^4\left(\sum_{i\ne k}{\bar x}_i{\bar x}_k +
\sum_{i=1}^n{\bar x}_i^2\right)\\
&=&b^4\left(\sum_{i=1}^n{\bar x}_i\right)^2\le b^4x_*^4n^2.
\end{eqnarray*}
Combining the above four displays yields
\begin{equation*}
  \Prob(\mathcal{A}^{c}) \leq \frac{4\sqrt{\log (2M)}
  \Big[(8\log(12M))^2+1\Big]^{1/2}}{\lom}\,.
\end{equation*}
\end{proof}

\section{Maximal moment inequality}
\label{sec:nem} In this section we prove the following inequality
for the $m$-th moment of maxima of sums of independent random
variables.

\begin{lemma}{\bf (Maximal moment inequality)}\label{nem-sara} Let $Z_1 , \ldots , Z_n$
be independent random vectors in $\R^M$, and let $Z_{i,j}$ denote
the $j$-th component of $Z_i$. Then for any $m \geq 1$ and $M \geq 1$
we have
$$\E \left(\max_{1 \leq j \leq M}
\left | \sum_{i=1}^n\biggl  ( Z_{i,j} - \E Z_{i,j}\biggr ) \right
|^m\right) \le \biggl [ 8 \log (c(m)M) \biggr ]^{m/2} \E
\left(\biggl [ \max_{1 \leq j \leq M} \sum_{i=1}^n Z_{i,j}^2 \biggr
]^{m/2}\right)\,, $$ where $c(m)=\min\{c>0:\, e^{m-1}-1\le (c-2)M
\}$. In particular, $2\le c(m)\le e^{m-1}+1$.
\end{lemma}

Before giving the proof, we make some comments. The case $m=2$ of
Lemma \ref{nem-sara} implies -- modulo constants -- Nemirovski's
inequality (see \cite{nem00}, page 188, and \cite{dumbgen2008nir},
Corollary 2.4). In general, Nemirovski's inequality concerns the
second moment of $\ell_p$-norms ($1 \leq p \leq \infty$) of sums of
independent random variables in $\R^M$, whereas we only consider
$p=\infty$. On the other hand, even for $m=2$ Lemma \ref{nem-sara}
is more general than what is given by Nemirovski's inequality
because we interchange the maximum and the sum on the right hand
side. The case $M=1$ of Lemma \ref{nem-sara} yields the
Marcinkiewicz-Zygmund inequality (see \cite{petrov}, page 82), and
as an immediate consequence the inequality
\begin{equation}\label{mz}
\E\left(\Big|\sum_{i=1}^n \xi_i\Big|^m\right) \le [ 8 \log (c(m))
]^{m/2} n^{m/2-1}\sum_{i=1}^n \E\left|\xi_i\right|^m, \quad m\ge2,
\end{equation}
for independent zero-mean random variables $\xi_i$. Thus, as a
particular instance, we give a short proof of (\ref{mz}) and provide
the explicit constant. This constant is of the optimal order in $m$
but larger than the one obtainable from the recent sharp moment
inequality due to Rio \cite{rio09}.
\begin{proof}
Let $(\varepsilon_1 , \ldots , \varepsilon_n)$ be a sequence of
i.i.d. Rademacher random variables independent of ${\bf Z} = (Z_1 ,
\ldots , Z_n) $. Let $\E_{\bf Z}$ denote conditional expectation
given $\bf Z$. By Hoeffding's inequality, for all $L >0$ and all $i$
and $j$,
\begin{equation}\label{hoeff}
\E_{\bf Z} \exp [ {Z_{i,j} \varepsilon_i / L} ] \leq \exp [ {
Z_{i,j}^2 / (2 L^2 )} ] .
\end{equation}
Define
\begin{equation*}
  \zeta = \max_{1 \leq j \leq M}
  \biggl | \sum_{i=1}^n Z_{i,j} \varepsilon_i \biggr
  |.
\end{equation*}
Using successively Jensen's inequality (the function $x\mapsto
\log^m\left(x+e^{m-1} - 1 \right)$ is concave for $x\ge 1$), the
inequality $e^{|x|}\le e^x + e^{-x}, \, \forall \ x\in \R,$ the
independence of $\varepsilon_i$, and (\ref{hoeff}), we obtain
\begin{eqnarray*}
\E_{\bf Z} (\zeta^m) &\leq& L^m \E_{\bf Z}\log^m \biggl \{ \exp
  \left [ \zeta/L \right ]   + {e}^{m-1}-1 \biggr \}\\
&\leq& L^m \log^m  \biggl \{  \E_{\bf Z}\exp
  \left [ \zeta/L \right ]   + {e}^{m-1}-1 \biggr \}
\\
&\leq& L^m \log^m  \biggl \{ \sum_{j=1}^M \E_{\bf Z}\exp
  \left [ \biggl | \sum_{i=1}^n Z_{i,j} \varepsilon_i \biggr
  |/L \right ]   + {e}^{m-1}-1 \biggr \}
\\
 &\leq& L^m \log^m  \left \{ 2M \exp \biggl [
 \max_{1 \leq j \le M}\sum_{i=1}^n  Z_{i,j}^2 /( 2 L^2 )
 \biggr ]
+ {e}^{m-1} -1 \right \}\,.
\end{eqnarray*}
Note that $2Mx+{e}^{m-1} -1\le c(m)Mx$ for all $x\ge 1$, where
$c(m)$ is the constant defined in the statement of the lemma. This
and the previous display yield
\begin{eqnarray*}
\E_{\bf Z} (\zeta^m) &\leq& L^m \log^m  \left \{ c(m)M  \exp \biggl
[
 \max_{1 \leq j \le M}\sum_{i=1}^n  Z_{i,j}^2 /( 2 L^2 )
 \biggr ]
 \right \}\\
&=& L^m \left \{ \log (c(m)M) + \frac{\max_{1 \leq j \leq M}
\sum_{i=1}^n
  Z_{i,j}^2}{2 L^2 } \right \}^m .
\end{eqnarray*}
Choosing
$$L= \sqrt{ \frac{\max_{1 \leq j \leq M }\sum_{i=1}^{n}
Z_{i,j}^2}{2 \log (c(m)M) }} $$ gives
$$\E_{\bf Z}
\left(\max_{1 \leq j \le M} \biggl | \sum_{i=1}^n Z_{i,j}
\varepsilon_i \biggr |^m\right) \leq \left [ 2 \log (c(m)M) \max_{1
\leqslant j \leq M}\sum_{i=1}^n  Z_{i,j}^2 \right ]^{m/2} .
$$ Hence,
$$\E \left(\max_{1 \leq j \leq M} \biggl |
\sum_{i=1}^n Z_{i,j} \varepsilon_i \biggr |^m\right) \leq \biggl[ 2
\log (c(m)M) \biggr]^{m/2} \E \left(\left [ \max_{1 \leq j \leq
M}\sum_{i=1}^n  Z_{i,j}^2  \right ]^{m/2}\right)\, .
$$
Finally, we de-symmetrize (see Lemma 2.3.1 page 108 in \cite{VW96}):
$$\left (
\E \max_{1 \leq j \leq M} \biggl | \sum_{i=1}^n (Z_{i,j} - \E
Z_{i,j}) \biggr |^m\right )^{1/m} \leq 2 \left ( \E \max_{1 \leq j
\leq M} \biggl | \sum_{i=1}^n Z_{i,j} \varepsilon_i \biggr |^m\right
)^{1/m}.
$$ 

\end{proof}


\subsection*{Acknowledgments}
Part of this work was supported by the IST Programme of the European
Community, under the PASCAL Network of Excellence, IST-2002-506778
as well as by the EPSRC Grant EP/D071542/1 and ANR ``Parcimonie".
\appendix
\section{Auxiliary results}

Here we collect some auxiliary results which we have use in the paper.

The first result is taken from \cite[Eq.~(27)]{cavalier} and was used in the
proof of Lemma \ref{lem:1}.
\begin{lemma}\label{chi2}
Let $\xi_1,\dots,\xi_N$ be i.i.d. $\GD(0,1)$, $v=(v_1,\dots,v_N) \neq
0$, $\eta_v = \frac{1}{\sqrt{2} \|v\|} \sum\limits_{i=1}^N (\xi_i^2
- 1)v_i$ and $m(v) =  \frac{\|v\|_\infty}{\|v\|}$. We have, for all
$x > 0$, that
\begin{equation*}
\Prob (|\eta_v| > x) \leq 2
\exp\left(-\frac{x^2}{2(1+ \sqrt{2} x m(v))}\right).
\end{equation*}
\end{lemma}

The next lemma provides the link between Assumptions \ref{MC} and
\ref{RE} and was used extensively in our analysis in Section
\ref{sec:model-selection}.
\begin{lemma}\label{lem:2}
Let Assumption \ref{MC} be satisfied. Then Assumption \ref{RE} is
 satisfied with $\kappa  = \sqrt{(1 - 1/\alpha)\phi}$.
\end{lemma}
\begin{proof}
We use here the notations introduced in the proof of Theorem
\ref{est-supnorm}. For any subset $J$ of $\N_M$ such that
$|J|\leq s$ and any $\Delta\in \R^{K}$ we have
\begin{align*}
  \left|\Delta_J\trans(\Psi-\phi I_{K\times K})\Delta_J\right| &\leq
  \sum_{j,j'\in
  J}\sum_{t=1}^{K_j}\sum_{t'=1}^{K_{j'}}
\left|\left(\tilde{\Psi}[j,j']\right)_{t,t'}\right||\tilde{\Delta}_{t}^{j}|
|\tilde{\Delta}_{t'}^{j'}|\\ &= \sum_{j,j'\in
J}\sum_{t=1}^{\min(K_j,K_{j'}}\left|\left(\tilde{\Psi}[j,j']\right)_{t,t}\right|
|\tilde{\Delta}_{t}^{j}||\tilde{\Delta}_{t}^{j'}|\\ &\hspace{3cm}+
\sum_{j,j'\in J}\sum_{t=1}^{K_j}\sum_{t'=1,t'\neq
t}^{K_{j'}}\left|\left(\tilde{\Psi}[j,j']\right)_{t,t'}\right|
|\tilde{\Delta}_{t}^{j}||\tilde{\Delta}_{t'}^{j'}|.
\end{align*}
We now treat separately the first and second terms in the right-hand
side of the above display. For the first term we have, using
consecutively Assumption \ref{MC}, Cauchy-Schwarz and Minkowski's
inequality for the Euclidean norm in $\mathbb R^{K_j}$, that
\begin{eqnarray*}
  \sum_{j,j'\in
J}\sum_{t=1}^{K_j}\left|\left(\tilde{\Psi}[j,j']\right)_{t,t}\right|
|\tilde{\Delta}_{t}^{j}||\tilde{\Delta}_{t}^{j'}|
&\leq& \frac{\lambda_{\min}\phi}{14\alpha
\lambda_{\max}s}\sum_{t=1}^{K_j}\left(\sum_{j\in
J}|\tilde{\Delta}_{t}^j|\right)^2\\
&\leq& \frac{\lambda_{\min}\phi}{14\alpha
\lambda_{\max}s}\|\Delta_{J}\|_{2,1}^2\\
&\leq& \frac{\lambda_{\min}\phi}{14\alpha
\lambda_{\max}}\|\Delta_{J}\|^2.
\end{eqnarray*}
For the second term we get, using Assumption \ref{MC} and
Cauchy-Schwarz's inequality twice, that
\begin{eqnarray*}
\sum_{j,j'\in J}\sum_{t=1}^{K_j}\sum_{t'=1,t'\neq
t}^{K_{j'}}\left|\left(\tilde{\Psi}[j,j']\right)_{t,t'}\right|
|\tilde{\Delta}_{t}^{j}||\tilde{\Delta}_{t'}^{j'}|
&\leq& \frac{\lambda_{\min}\phi}{14\alpha\lambda_{\max}s}\left(
\sum_{j\in J}\frac{1}{\sqrt{K_j}}\sum_{t=1}^{K_j}|\Delta^j_{t}|
\right)^2\\
&\leq& \frac{\lambda_{\min}\phi}{14\alpha
\lambda_{\max}}\|\Delta_{J}\|^2.
\end{eqnarray*}

Combining the two above displays yields
\begin{eqnarray*}\label{stanarg-chapt-4}
\frac{\Delta_{J}\trans\Psi\Delta_{J}}{\|\Delta_{J}\|^{2}} &=& \phi+
\frac{\Delta_J\trans(\Psi-\phi I_{K\times
K})\Delta_J}{\|\Delta_J\|^2}\\
&\geq& \phi\left(1-\frac{2\lambda_{\min}}{14\alpha \lambda_{\max}
}\right).
\end{eqnarray*}

We proceed similarly to treat the quantity $|\Delta_{J^c}\Psi
\Delta_{J}|$. We have, using Assumption \ref{MC}, Cauchy-Schwarz and
Minkowski's inequalities, that
\begin{align*}
  |\Delta_{J^c}\Psi \Delta_{J}| &\leq \sum_{j\in J^c,j'\in
J}\sum_{t=1}^{K_j}\left|\left(\tilde{\Psi}[j,j']\right)_{t,t}\right|
|\tilde{\Delta}_{t}^{j}||\tilde{\Delta}_{t}^{j'}|\\
&\hspace{3cm}+ \sum_{j\in J^c,j'\in
J}\sum_{t=1}^{K_j}\sum_{t'=1,t'\neq
t}^{K_{j'}}\left|\left(\tilde{\Psi}[j,j']\right)_{t,t'}
\right||\tilde{\Delta}_{t}^{j}||\tilde{\Delta}_{t'}^{j'}|\\
&\leq \frac{\lambda_{\min}\phi}{14\alpha
\lambda_{\max}s}\|\Delta_{J^c}\|_{2,1}\|\Delta_J\|_{2,1}\\
&\hspace{3cm}+ \frac{\lambda_{\min}\phi}{14\alpha
\lambda_{\max}s}\left( \sum_{j\in
J} \sum_{t=1}^{K_j}\frac{1}{\sqrt{K_j}}|\Delta^j_{t}|
\right)\left( \sum_{j\in
J^c}\sum_{t=1}^{K_j}\frac{1}{\sqrt{K_j}}
|\Delta^j_{t}| \right)\\
&\leq \frac{2\lambda_{\min}\phi}{14\alpha
\lambda_{\max}s}\|\Delta_J\|_{2,1} \|\Delta_{J^c}\|_{2,1}.
\end{align*}

Next we have, for any vector $\Delta \in \mathbb R^K$ satisfying the
inequality $\sum_{j\in J^{c}}\lambda_j\|\Delta^{j}\|\leq 3\sum_{j\in
J}\lambda_j\|\Delta^{j}\|$, that
\begin{eqnarray*}
\|\Delta_{J^c}\|_{2,1} &=& \sum_{j\in J^c}\|\Delta^j\|\\
&\leq& \sum_{j\in J^c}\frac{\lambda_{j}}{\lambda_{\min}}
\|\Delta^j\|\\
&\leq& \frac{3}{\lambda_{\min}} \sum_{j\in
J}\lambda_j\|\Delta^{j}\|\\
&\leq&\frac{3\lambda_{\max}}{\lambda_{\min}}\|\Delta_{J}\|_{2,1}.
\end{eqnarray*}

Combining these inequalities we find that
\begin{eqnarray*}
\frac{\Delta \trans \Psi \Delta}{\|\Delta_{J}\|^{2}} &\geq&
\frac{\Delta_{J}\trans\Psi\Delta_{J}}{\|\Delta_{J}\|^{2}} +
\frac{2\Delta_{J^{c}}\trans\Psi\Delta_{J}}{\|\Delta_{J}\|^{2}}\\
&\geq& \phi - \frac{2\lambda_{\min}\phi}{14\alpha \lambda_{\max
}} - \frac{12\phi\|\Delta_{J}\|_{2,1}^2}{14\alpha s \|\Delta_{J}\|^2}\\
&\geq& \left(1- \frac{1}{\alpha}\right)\phi.
\end{eqnarray*}
\end{proof}

\begin{lemma}\label{ineq_p}
Let $Ts\ge 8$. If $\omega$ and $\omega'$ are two elements of ${\cal
N}'$ such that $\rho'(\omega,\omega') \geq \frac{Ts}{8}$, then the
cardinality of the set $J(\omega,\omega')= \left\{j\le s:
\sum_{t=1}^{T}I\{\omega_{tj} \neq \omega'_{tj}\}>\frac{T}{16}
\right\}$ is greater than or equal to $\frac{s}{16}$.
\end{lemma}
\begin{proof}
Assume that $|J(\omega,\omega')|< s/16$. Then, denoting by
$J(\omega,\omega')^c$ the complement of $J(\omega,\omega')$, and
using that $|J(\omega,\omega')^c|\le s$, we get
$$
\rho'(\omega,\omega')\le \sum_{j\in
J(\omega,\omega')^c}\sum_{t=1}^{T}I\{\omega_{tj} \neq \omega'_{tj}\}
+ |J(\omega,\omega')|T< Ts/8,
$$
which contradicts the premise of the lemma.
\end{proof}

\end{document}